\documentclass{amsart}

\usepackage{amsfonts,amsthm,amsmath,amssymb}
\usepackage{mathtools,verbatim,mathrsfs,bbm}

\usepackage[dvipsnames]{xcolor}
\usepackage{soul}
\usepackage{hyperref}
\usepackage[capitalise]{cleveref} 
\usepackage{enumitem,tabularx,array,booktabs}
\usepackage{ulem}
\usepackage{bibentry}
\makeatletter
\@namedef{subjclassname@2020}{
  \textup{2020} Mathematics Subject Classification}
\makeatother

\newtheorem{theorem}{Theorem}[section]

\newtheorem{lemma}[theorem]{Lemma}
\newtheorem{claim}{Claim}
\newtheorem{proposition}[theorem]{Proposition}

\newtheorem{corollary}[theorem]{Corollary}
\newtheorem{question}{Question}

\newtheorem*{theorem*}{Theorem}
\newtheorem*{maintheorem*}{Main Theorem}
\newtheorem*{lemma*}{Lemma}
\newtheorem*{techlemma*}{Technical Lemma}
\newtheorem*{claim*}{Claim}
\newtheorem*{conjecture*}{Conjecture}
\newtheorem*{question*}{Question}
\newtheorem*{corollary*}{Corollary}
\newtheorem*{maincorollary*}{Main Corollary}
\newtheorem*{idea*}{Idea}
\newtheorem*{fact*}{Fact}

\theoremstyle{plain} 
\newcommand{\thistheoremname}{}
\newtheorem*{genericthm}{\thistheoremname}
\newenvironment{namedthm}[1]
  {\renewcommand{\thistheoremname}{#1}%
   \begin{genericthm}}
  {\end{genericthm}}

\theoremstyle{plain}
\newcounter{MainTheoremCounter}

\newtheorem{Maintheorem}[MainTheoremCounter]{Theorem}

\theoremstyle{definition}
\newtheorem{definition}[theorem]{Definition}
\newtheorem{remark}[theorem]{Remark}
\newtheorem{example}[theorem]{Example}

\newcommand{\N}{\mathbb{N}}
\newcommand{\Z}{\mathbb{Z}}
\newcommand{\R}{\mathbb{R}}
\newcommand{\C}{\mathbb{C}}
\newcommand{\Q}{\mathbb{Q}}

\newcommand{\T}{\mathbb{T}}

\newcommand{\folner}{F{\o}lner}

\newcommand{\eps}{\epsilon}

\newcommand{\bz}{\text{Bohr$_0$}}
\newcommand{\bzs}{\text{Bohr$_0^*$}}

\newcommand{\id}{\text{Id}}

\newcommand{\defeq}{\vcentcolon=}

\newcommand{\diam}{\text{diam}}

\newcommand{\dyret}{\mathfrak{R}}

\newcommand{\rotpi}{\rotatebox[origin=c]{180}{$\pi$}}
\newcommand{\secquest}{Related}

\newcommand\restr[2]{{ \left.\kern-\nulldelimiterspace #1 \right|_{#2}}}

\definecolor{orange}{rgb}{1,0.5,0}

\title{On Katznelson's Question for skew product systems}

\author[D.\ Glasscock]{Daniel Glasscock}
\address[Daniel Glasscock]{Mathematical Sciences Department \\ University of Massachusetts Lowell\\
Lowell, MA, USA}
\email{daniel\textunderscore glasscock@uml.edu}

\author[A.\ Koutsogiannis]{Andreas Koutsogiannis}
\address[Andreas Koutsogiannis]{Department of Mathematics, Aristotle University of Thessaloniki, Thessaloniki, 54124, Greece}
\email{akoutsogiannis@math.auth.gr}

\author[F.\ K.\ Richter]{Florian~K.~Richter}
\address[Florian Karl Richter]{Institute of Mathematics\\ {\'E}cole Polytechnique F{\'e}d{\'e}rale de Lausanne (EPFL)\\
Lausanne, Vaud, Switzerland}
\email{f.richter@epfl.ch}

\begin{document}

\begin{abstract}
    Katznelson's Question is a long-standing open question concerning recurrence in topological dynamics with strong historical and mathematical ties to open problems in combinatorics and harmonic analysis.
    In this article, we give a positive answer to Katznelson's Question for certain towers of skew product extensions of equicontinuous systems, including systems of the form $(x,t) \mapsto (x + \alpha, t + h(x))$. 
    We describe which frequencies must be controlled for in order to ensure recurrence in such systems, and we derive combinatorial corollaries concerning the difference sets of
    syndetic subsets of the natural numbers.
\end{abstract}



\maketitle
\tableofcontents

\clearpage

\section{Introduction}

Recurrence is a central topic in the theory of dynamical systems that concerns the fundamental question of how and when a point or set recurs to its initial position.
This paper addresses Katznelson's Question, a long-standing open problem concerning recurrence in topological dynamics with strong historical and mathematical ties to open problems in combinatorics and harmonic analysis.

\subsection{Main results}
\label{sec_overview}

A \emph{topological dynamical system} (henceforth, a \emph{system}) is a pair $(X,T)$, where $(X,d_X)$ is a compact metric space and $T: X \to X$ is a continuous map. A set $R \subseteq \N$ of positive integers is a \emph{set of recurrence for the system $(X,T)$} if there exists a point $x \in X$ that returns arbitrarily closely to its initial position at times in $R$, that is, $\inf_{n \in R} d_X(x,T^nx) = 0$. The set $R$ is a \emph{set of topological recurrence} if it is a set of recurrence for all systems.
Because the phase space $X$ of any system $(X,T)$ is compact, it is easy to see, for example, that $\N$ is a set of recurrence.
More involved examples include the set of positive differences $\{b-a \ | \ a,b\in E,~b>a\}$  of any infinite set $E \subseteq \N$ and the set of squares $\{n^2 \ | \ n\in\N\}$.\footnote{Both sets are known more generally to be \emph{sets of measurable recurrence}; see \cite[Theorems 3.16 and 3.18]{furstenberg_book_1981}.}
Sets that are not sets of recurrence include, for example, all sets that are not \emph{divisible} and all sets that are \emph{lacunary}.\footnote{A set $S\subseteq\N$ is \emph{divisible} if for all $q\in\N$, there exists $s\in S$ such that $s\equiv 0\pmod q$. Writing $S=\{s_1 < s_2<\ldots\}$, the set $S$ is \emph{lacunary} if it has ``exponential growth'' in the sense that $\liminf_{k\to\infty}s_{k+1}/s_k>1$.}

The simplest examples of non-trivial systems are rotations of the $d$-dimensional torus $\T^d = \R^d / \Z^d$, given by $T: x \mapsto x + \alpha$, where $\alpha \in \T^d$.  A \emph{set of Bohr recurrence} is a set of recurrence for all finite-dimensional toral rotations.  
By definition, any set of topological recurrence is also a set of Bohr recurrence. Since rotations on finite-dimensional tori comprise a narrow subclass of topological dynamical systems, one is led to expect that sets of topological recurrence comprise a narrow subclass of the sets of Bohr recurrence. The extent to which this is true remains an important unsolved problem, one that was popularized in the dynamics community by Katznelson \cite{KatznelsonChromaticNumber2001}. This question -- and its various equivalent formulations to which we turn in a moment -- is the main subject of our study.

\begin{namedthm}{Katznelson's Question}
    Is every set of Bohr recurrence a set of topological recurrence?
\end{namedthm}

Not only is Katznelson's Question open, there seems to be no consensus among experts as to the expected answer.
There are very few concrete examples of sets which could provide a negative answer: sets which are known to be sets of Bohr recurrence but whose other dynamical recurrence properties are unknown (see Grivaux and Roginskaya \cite{Grivaux_Roginskaya_2013} and Frantzikinakis and McCutcheon \cite[Future Directions]{frantzikinakis_mccutcheon_recurrence_2012}).
The situation does not look more promising in the opposite direction: a positive answer to Katznelson's Question is known only in a few special cases. For example, it was shown recently in \cite{HostKraMaass2016} that sets of Bohr recurrence are sets of recurrence for nilsystems, a class of systems of algebraic origin that generalize rotations on tori.

A natural next step towards a resolution of Katznelson's Question is to consider skew product extensions of equicontinuous systems.
In the structure theory of topological dynamical systems initiated by Furstenberg and Veech, such systems represent a single step up in complexity from toral rotations (see \cite{glasner_structure_theory_2000}). 
The 2-torus transformation $(x,y) \mapsto (x+\alpha, y+h(x))$, where $h\colon\T\to\T$ is continuous, is a simple example of a skew product extension of the 1-torus rotation $x \mapsto x+ \alpha$ for which Katznelson's Question has thus far been unresolved. Our main contribution in this paper is a positive answer to Katznelson's Question for a class of towers of skew product extensions over equicontinuous systems that includes this example and others.

\begin{Maintheorem}
\label{maintheorem_katznelson_for_skew_towers}
Every set of Bohr recurrence is a set of recurrence for skew product systems of the form $(X \times \T^{d}, T_{\vec h})$, where
\[T_{\vec h}(x, t_1, \ldots, t_d) \defeq \big(Tx, t_1 + h_1(x), t_2 + h_2(t_1), \ldots, t_d + h_d(t_{d-1}) \big),\]
the system $(X,T)$ is equicontinuous, and $h_1: X \to \T$, $h_2, \ldots, h_d: \T \to \T$ are continuous maps.
\end{Maintheorem}

Katznelson's Question asks whether or not recurrence along a set $R \subseteq \N$ is guaranteed by ensuring recurrence along $R$ in \emph{all} finite dimensional toral rotations.
Thus, a positive answer to Katznelson's Question for the system $(X,T)$ begs the finer question: \emph{which} rotations suffice to ensure recurrence along $R$ in the system $(X,T)$?
In the course of our investigation, we identify the frequencies that participate in the recurrence behaviour of towers of skew product extensions of the form described in \cref{maintheorem_katznelson_for_skew_towers}. 
Surprisingly, we find that in such systems, it is not enough to control for the frequencies inherent to the base equicontinuous system. This finding is in contrast with the behaviour previously observed in other types of systems for which Katznelson's Question has been answered in the affirmative, such as nilsystems \cite{HostKraMaass2016}. 

More precisely, the following theorem demonstrates that in addition to the frequencies inherent to the base equicontinuous system, it is necessary to control for new frequencies introduced by the extensions to ensure recurrence. In particular, new frequencies can be introduced even when the extensions do not increase the size of the largest equicontinuous factor of the system.
For skew-product extensions of equicontinuous systems, the new frequencies introduced are the means of the skewing functions, as described in more detail in \cref{remark_details_on_hidden_freq}.

\begin{Maintheorem}
\label{maintheorem_hidden_frequencies}
There exists an irrational toral rotation $(\T, T)$ and a continuous map $h: \T \to \T$ for which the skew product system $(\T^2,T_h)$,
\[T_h(x,t) \defeq \big(Tx, t + h(x) \big),\]
satisfies the following:
\begin{enumerate}
    \item \label{item_minimal_and_correct_kronecker} $(\T^2,T_h)$ is minimal and its largest equicontinuous factor is $(\T,T)$;
    \item there exists a set of recurrence for $(\T,T)$ that is not a set of recurrence for $(\T^2,T_h)$.
\end{enumerate}
\end{Maintheorem}

There are several next steps suggested by Theorems~\ref{maintheorem_katznelson_for_skew_towers} and~\ref{maintheorem_hidden_frequencies}; we record many of them as open problems in \cref{sec_open_questions}. Questions \ref{question_basic_skew_over_2_step_nil} and \ref{question_more_general_isometric_ext} in \cref{sec_next_steps} feature some simple examples of extensions of equicontinuous systems for which an answer to Katznelson's Question is still not known. An understanding of general isometric extensions -- ones which generalize skew product extensions -- from the point of view of recurrence would represent a major step toward resolving Katznelson's Question for general distal systems. It would also direct attention toward weak mixing systems at the other end of the dynamical spectrum as the next class to analyze from this perspective.

Katznelson's Question and its relatives were considered in equivalent, combinatorial forms long before they were popularized in dynamical terms. A subset of $\N$, respectively $\Z$,  is \emph{syndetic} (more traditionally, \emph{relatively dense}) if finitely many of its translates cover $\N$, respectively $\Z$.  A \emph{Bohr neighborhood of zero} is a set of the form
\begin{align}
\label{eqn_bohr_neighborhood_of_zero}
    \big\{ n \in \Z \ \big| \ \| n \alpha \| < \delta \big\}, \quad \alpha \in \T^d, \ \delta > 0,
\end{align}
where $\| \cdot \|$ denotes the Euclidean distance to zero on $\T^d$.  Bohr neighborhoods of zero and their translates are syndetic sets that generate the \emph{Bohr topology} on $\Z$, the coarsest topology on the integers with respect to which all trigonometric polynomials are continuous. Katznelson's Question is equivalent to the following one, in the sense that a positive answer to one yields a positive answer to the other.

\begin{namedthm}{Katznelson's Question (combinatorial form)}
    If $A \subseteq \N$ is syndetic, does the set of pairwise differences
    \[A-A \defeq \big\{ a_1-a_2 \ \big| \ a_1, a_2\in A
    \big\}
    \]
contain a Bohr neighborhood of zero?
\end{namedthm}

As with the dynamical formulation, there are only a handful of special cases in which a positive answer is known.  We show in \cref{theorem_two_coloring_Katznelson}, for example, that if two translates of $A$ cover $\N$, then $A-A$ contains 
$d \Z=\{dn:n\in\Z\}$, a (periodic) Bohr neighborhood of zero.

Katznelson's Question also finds a useful formulation in terms of $1$-torus-valued sequences. We demonstrate the equivalence between Katznelson's Question and the following one in \cref{sec_comb_forms_of_katznelson}.

\begin{namedthm}{Katznelson's Question (sequential form)}
    Is it true that for all $f: \Z \to \T$ and all $\eps > 0$, the set
    \begin{align}
    \label{eqn_set_in_sequences_question}
        \big\{ m \in \Z \ | \ \inf_{n \in \Z} \| f(n+m) - f(n) \| < \eps \big\}
    \end{align}
    contains a Bohr neighborhood of zero?
\end{namedthm}

A sequence $f: \Z \to \T$ is \emph{Bohr almost periodic on $\Z$} if for all $\eps > 0$, the set
    \[\big\{ m \in \Z \ | \ \sup_{n \in \Z} \| f(n+m) - f(n) \| < \eps \big\}\]
contains a Bohr neighborhood of zero. By definition, the sequential form of Katznelson's Question has a positive answer for almost periodic sequences on $\Z$; \cref{maintheorem_combinatorial_form} shows that the question has a positive answer for those sequences $f$ whose discrete derivative $\Delta_1 f(n) = f(n+1) - f(n)$ is almost periodic. In fact, the result applies more generally to any sequence that becomes almost periodic after finitely many discrete derivatives.  It also provides a class of syndetic subsets of $\N$ whose pairwise differences contain a Bohr neighborhood of zero; see \cref{example_application_of_thm_c}.

\begin{Maintheorem}
\label{maintheorem_combinatorial_form}
Let $f: \Z \to \T$ and $k \in \N \cup \{0\}$. If the $k^{\text{th}}$ discrete derivative of $f$, $\Delta_1^k f$, is Bohr almost periodic, then for all $\eps > 0$, the set
\[A \defeq \big\{ n \in \N \ \big| \ \| f(n) - f(0) \| < \eps \big\}\]
is syndetic and its set of pairwise differences, $A-A$, contains a Bohr neighborhood of zero.  In particular, for any such $f$ and any $\eps>0$, the set in \eqref{eqn_set_in_sequences_question} contains a Bohr neighborhood of zero.
\end{Maintheorem}

We move next to recount the history behind Katznelson's Question and its relatives.

\subsection{History and context}
\label{sec_history}

A storied theorem of Steinhaus \cite{Steinhaus1920} gives that the set of differences $X-X$ of a set $X \subseteq \R$ of positive Lebesgue measure contains an open neighborhood of zero.  Weil \cite{weil1940} extended the result to locally compact groups with respect to the Haar measure.  It is natural to ponder the extent to which analogues of Steinhaus' result may hold in other settings.  This becomes particularly interesting in the context of the integers, where a natural topology, the Bohr topology, is generated by all sets of the form given in \eqref{eqn_bohr_neighborhood_of_zero}. Thus, the combinatorial form of Katznelson's Question can be understood as an analogue to Steinhaus' result concerning the Bohr topology on $\Z$.

A more historically motivated impetus for Katznelson's Question begins with the work of Bogolyubov \cite{bogolyubov1939}, who was one of the first to explore the relationship between difference sets and Bohr almost periodic functions. 
In the process of giving a new proof of Bohr's characterization of almost periodic functions\footnote{A function $f: \R \to \C$ is \emph{Bohr almost periodic} if for all $\eps > 0$ there exists $L > 0$ such that the set $\{m \in \R \ | \ \sup_{x \in \R} | f(x+m) - f(x) | < \eps \}$ has non-empty intersection with every interval in $\R$ of length at least $L$.} on $\R$ (as those uniformly approximable by trigonometric polynomials), he proved that if $A \subseteq \Z$ has \emph{positive upper asymptotic density}, i.e.,
\[\limsup_{N \to \infty} \frac{\big| A \cap [-N, N] \big|}{2N+1} > 0,\] 
then the set $(A-A)-(A-A)$ contains a Bohr neighborhood of zero. Bogolyubov's work seeded a vast array of generalizations to other settings, including non-abelian, non-amenable, and non-discrete ones; \cref{table_discrete_results} organizes many of the main results in discrete settings.  To focus the narrative in this section, we will concentrate primarily on those results which have advanced our understanding in the integers.

Bogolyubov's consideration of density ties the history of Katznelson's Question inextricably to the history of the following, related question.

\begin{namedthm}{\secquest{} Question}
    If $A \subseteq \Z$ has positive upper asymptotic density, does its set of pairwise differences, $A-A$, contain a Bohr neighborhood of zero?
\end{namedthm}

\noindent K\v{r}\'{i}\v{z} \cite{kriz1987} gave a negative answer to the \secquest{} Question; we recount some of that history in more detail below.  The historical bond between Katznelson's Question and the \secquest{} Question is so tight that it is not possible to recount the history of one without an equal treatment of the other.

\folner{} \cite{folnerbogoliouboff1954, folnernoteonbogoliouboff1954} proved that the set $(A-A)-(A-A)$ contains a Bohr neighborhood of zero for any set $A$ of positive upper Banach density, ie.,
\begin{align}
    \label{eqn_upper_banach_density}
    \limsup_{N \to \infty} \max_{z \in \Z} \frac{\big| A \cap [z+1,z+N] \big|}{N} = \sup_{\lambda} \lambda(A) > 0,
\end{align} 
where the supremum is over the set $\lambda$ of left-translation invariant means (positive linear functionals of norm 1) on the bounded, real-valued functions on $\Z$.\footnote{\label{note_equiv_of_densities} It is shown in \cite[Theorem 2.2a]{peres_application_of_banach_means_1988} that the two quantities in \eqref{eqn_upper_banach_density} are equal for subsets of $\N$. In $\Z$, \folner{}'s result applies to sets of positive ``upper Weyl mean measure,'' which is shown in \cite[Section 3]{bergelson_glasscock_interplay_2020} to be the same as the upper Banach density, even in more general groups and semigroups. Despite the fact that asymptotic density and Banach density are different, it can be shown that when considering the set of differences $A-A$, there is no difference between assuming that $A$ has positive upper asymptotic density and assuming that $A$ has positive upper Banach density; see \cite[Theorem 3.20]{furstenberg_book_1981}.}
\folner{}  also proved that when $A$ has positive upper Banach density, the set $A-A$ contains a Bohr neighborhood of zero up to a set of exceptions of zero Banach density.  Veech \cite[Theorem 4.1]{veech1968}, following \folner{}'s argument, arrived at the same conclusion when $A$ is syndetic; Veech's argument works verbatim for sets of positive upper Banach density. The only apparent difference between \folner{}'s theorem and Veech's is in their definitions of density : \folner{} uses the upper Weyl mean measure, while Veech considers a supremum of the values assigned by translation invariant means. We know today that these those notions of density are exactly the same (cf. footnote~\ref{note_equiv_of_densities}).

Though the combinatorial form of Katznelson's Question and the \secquest{} Question would have been natural to anyone interested in this thread of results, it seems that neither appeared explicitly in print for some time. As far as we know, the combinatorial form of Katznelson's Question appears first in the literature as part of a more general program in Landstad \cite[Page 214]{landstad1971}:
\begin{quote}
    \emph{``For an amenable topological group, let $n$ be the minimal number such that $V^n$ is a Bohr neighbourhood whenever $V$ is a symmetric, relatively dense neighbourhood of $e$. We have seen that in general $n \leq 7$, $n \leq 5$ for abelian groups and $n \leq 4$ for discrete groups. A natural question is whether this number can be reduced for some special groups.''}
\end{quote} The first explicit mention of the \secquest{} Question appears to be due to Ruzsa \cite[Page 18.08]{ruzsa1982}, who attributes the question to personal communication with Flor.

It was Bochner who implicitly, if not explicitly, forged the connection between Bohr almost periodic functions and topological dynamics; see \cite[Chapter 4]{petersen_ergodic_theory} and \cite[Chapter 2]{weiss_single_orbit_book} for modern accounts.
In his characterization of the equicontinuous structure relation, Veech \cite{veech1968} drew a connection between the combinatorial form of Katznelson's Question and recurrence. Ellis and Keynes \cite{elliskeynes1972} and McMahon \cite{mcmahon_weak_disjointness1978} strengthened and generalized that connection by using tools from topological dynamical structure theory to show that $A-A+A-a$ contains a Bohr neighborhood of zero for ``many'' $a \in A$ when $A$ is syndetic. Ellis and Keynes seem to be the first to prove asymmetrical results along these line, showing in particular that the set $A-B+C$ contains a Bohr neighborhood of zero when $A$, $B$, and $C$ are members of the same minimal idempotent ultrafilter. 
More recently, Bergelson and Ruzsa \cite{bergelson_ruzsa_2009} showed that the triple sumset $r \cdot A+s \cdot A+t \cdot A$ contains a Bohr neighborhood of zero when $r$, $s$, and $t$ are integers with $r+s+t=0$ and $A$ is a set of positive upper asymptotic density; stronger results were achieved in \cite{le_le_2021} under the assumption that $A$ is syndetic. Uniformity in the dimension and diameter of Bohr sets contained in triple sumsets was recently demonstrated in broad generality by Bj\"{o}rklund and Griesmer \cite{bjorklund_griesmer_2019}.

Ruzsa \cite{ruzsa1982,ruzsa_difference_sets_unpublished_1985} formulated both the combinatorial form of Katznelson's Question and the \secquest{} Question and improved on Ellis and Keynes' result by showing that $A-A+A-a$ contains a Bohr neighborhood of zero for many $a \in A$ when $A$ is a set of positive upper asymptotic density.  (While \cite{ruzsa_difference_sets_unpublished_1985} was never published, several of the results appear in \cite{hegyvariruzsa2016}.) 
Ruzsa also expounded on a theorem of K\v{r}\'{i}\v{z} \cite{kriz1987}  that answers the \secquest{} Question in the negative: there exists a set of positive upper asymptotic density $A$ whose set of differences $A-A$ does not contain a Bohr neighborhood of zero. 
This result was recently strengthened by Griesmer \cite{griesmer_separating_bohr_denseness_2020}, answering a question in \cite[Page 196]{geroldinger_ruzsa_book_2009}: there exists a set of positive upper asymptotic density $A$ whose set of differences $A-A$ does not contain a translate of any Bohr neighborhood of zero.

The first more recent mention of the combinatorial form of Katznelson's Question in print is found in Glasner \cite{glasner1998}, who connected the problem to fixed points of actions of minimally almost periodic groups.  He shows that for a negative answer to Katznelson's Question, it suffices to construct a minimally almost periodic Polish monothetic topological group that acts with no fixed points by homeomorphisms on a compact space.  For a collection of related problems, see Pestov \cite{pestov_forty_questions2007}.

Katznelson \cite{KatznelsonChromaticNumber2001} was perhaps the first to explicitly formulate the eponymous question as one about recurrence in topological dynamics and is credited for popularizing this question in the dynamics community. Bergelson, Furstenberg, and Weiss \cite{BFW2006} employed tools and techniques from ergodic theory to prove, among other results, an asymmetric result reminiscent of \folner{}'s:  if $A, B \subseteq \Z$ have positive upper Banach density, then $A+B$ contains the intersection of a translate of a Bohr neighborhood of zero with a set containing arbitrarily long intervals; Griesmer \cite{griesmer_sumsets_of_dense_and_sparse_2012} improved on this result by weakening the positive Banach density assumption on one of the sets. Boshernitzan and Glasner \cite{boshernitzanglasner2009} summarized what is known about the Katznelson's Question and other related questions in the framework of dynamics and recurrence, and Huang, Shao, and Ye \cite{huang_shao_ye_nilbohr2016} formulated higher-order analogues of the Katznelson's Question in the framework of nilsystems and nil-Bohr sets.

Some of the most recent progress on the Katznelson's Question was made by Host, Kra, and Maass, who gave a positive answer for nilsystems and their proximal extensions (see \cite[Theorem 4.1 and Proposition 3.8]{HostKraMaass2016}).  Nilsystems are translations of compact homogeneous spaces of nilpotent Lie groups; Host, Kra, and Maass showed that in a minimal nilsystem, any set of recurrence for the largest equicontinuous factor is a set of recurrence for the nilsystem.  They also showed that if $(W,T) \to (X,T)$ is a proximal extension\footnote{\label{footnote_proximal_definition} An extension $\pi: (X,T) \to (Y,T)$ is \emph{proximal} if for all $x, y \in X$ with $\pi x = \pi y$ and all $\eps  >0$, there exists $n \in \N$ such that $d_X(T^n x, T^n y) < \eps$.} of minimal systems, then every set of recurrence for $(X,T)$ is a set of recurrence for $(W,T)$.

Host, Kra, and Maass's results combine with ours to give a list of systems in which a positive answer to Katznelson's Question is known: nilsystems, skew product extensions of equicontinuous systems by 1-tori, systems which support a measure with respect to which the transformation exhibits mixing on the $L^2$-orthocomplement of the Kronecker factor, and inverse limits, proximal extensions, and factors of such systems.  Beyond a few other sporadic examples, to our knowledge, this is a complete list.

\begin{table}[h]
\setlength\tabcolsep{11pt}
\setlength{\extrarowheight}{4pt}
\begin{tabular}{l|l|l}
\multicolumn{1}{c|}{Expression} & \multicolumn{1}{c|}{Syndeticity} & \multicolumn{1}{c}{Density} \\ \hline \hline

\begin{tabular}[c]{@{}l@{}} $A-A$ \end{tabular} & \begin{tabular}[c]{@{}l@{}}\cite{KatznelsonChromaticNumber2001}, \\ \cite{HostKraMaass2016} (nilsystems)\end{tabular} & \begin{tabular}[c]{@{}l@{}}\cite{ruzsa_difference_sets_unpublished_1985}, \cite{kriz1987}, \\ \cite{griesmer_separating_bohr_denseness_2020} \\[3pt] \end{tabular} \\ \hline

\begin{tabular}[c]{@{}l@{}} $A+B$ \\[3pt] \end{tabular} &  & \begin{tabular}[c]{@{}l@{}} \cite{BFW2006}, \cite{griesmer_sumsets_of_dense_and_sparse_2012} \\[3pt] \end{tabular} \\ \hline

\begin{tabular}[c]{@{}l@{}} $A+B+C$ \\[3pt] \end{tabular} & \begin{tabular}[c]{@{}l@{}}\cite{elliskeynes1972} (abelian), \\ \cite{BFW2006} \\[3pt] \end{tabular}  & \\ \hline

\begin{tabular}[c]{@{}l@{}}  $r \cdot A + s \cdot A + t \cdot A$ \\[3pt] \end{tabular} & \cite{le_le_2021} (abelian) & \cite{bergelson_ruzsa_2009} \\ \hline

\begin{tabular}[c]{@{}l@{}} $(A-A)-(A-a)$ \\[3pt] \end{tabular}  & \begin{tabular}[c]{@{}l@{}}\cite{elliskeynes1972} (abelian), \\ \cite{mcmahon_weak_disjointness1978} (abelian) \\[3pt] \end{tabular} & \begin{tabular}[c]{@{}l@{}}\cite{ruzsa1982}, \cite{ruzsa_difference_sets_unpublished_1985},\\ \cite{hegyvariruzsa2016}  \\[3pt] \end{tabular} \\ \hline

\begin{tabular}[c]{@{}l@{}} $(A-A)-(B-b)$ \\[3pt] \end{tabular}  & & \begin{tabular}[c]{@{}l@{}}  \cite{bjorklund_griesmer_2019} (amenable) \\[3pt] \end{tabular} \\ \hline

\begin{tabular}[c]{@{}l@{}} $(A-A)-(A-A)$ \\[3pt] \end{tabular}  & \begin{tabular}[c]{@{}l@{}}\cite{folnerbogoliouboff1954} (abelian), \\ \cite{veech1968} (abelian), \\ \cite{mcmahon_weak_disjointness1978} (amenable) \\[3pt]\end{tabular} & \begin{tabular}[c]{@{}l@{}}\cite{bogolyubov1939}, \\ \cite{folnernoteonbogoliouboff1954} (abelian), \\ \cite{landstad1971} (amenable)\end{tabular}  \\ \hline

\begin{tabular}[c]{@{}l@{}} $(A-A)-(B-B)$ \\[3pt] \end{tabular}  & \begin{tabular}[c]{@{}l@{}} \cite{elliskeynes1972} (non-abelian) \\[3pt]\end{tabular} & \\ \hline

\begin{tabular}[c]{@{}l@{}}$\big((A-A)-(A-A) \big)$\\ $ \ \ \ - \big((A-A)-(A-A) \big)$\end{tabular} & \begin{tabular}[c]{@{}l@{}}\cite{folner_ap_functions_on_abelian_groups} (abelian), \\ \cite{folner_proof_of_main_ap_theorem_1949} (abelian) \\[3pt]\end{tabular}  & \\
\end{tabular}

~\\~\\[-5pt]

\caption{A survey of the literature containing results pertaining to the combinatorial form of Katznelson's Question.  The columns separate the works based on the primary assumptions on the sets $A$, $B$, and $C$, while the parentheses indicate the setting: the sets $A$, $B$, and $C$ are subsets of the integers, abelian groups, amenable groups, or non-abelian groups (if no setting is indicated, the integers are the primary setting). The phrase ``nilsystems'' means that the primary results concern sets $A$ of the form $\{ n \in \N \ | \ T^n x \in U \}$, where $(X,T)$ is a nilsystem.}
\label{table_discrete_results}
\end{table}

\subsection{Outline of the article}
The article is organized as follows.  In \cref{sec_dynamical_prerecs}, we lay out the notation, terminology, and results from topological dynamics required for our main theorems.  \cref{sec_recurrence_in_skew_towers} covers basic results about skew products; Theorems \ref{maintheorem_katznelson_for_skew_towers} and \ref{maintheorem_hidden_frequencies} are proved in Sections \ref{sec_returns_in_towers} and \ref{sec_hidden_frequencies}, respectively.  In \cref{sec_combinatorial_setting}, we elaborate on Katznelson's Question and its consequences in a combinatorial setting, including a proof of \cref{maintheorem_combinatorial_form} in \cref{sec_proof_of_comb_form}.  We end the paper with \cref{sec_open_questions} by discussing a number of open questions and directions.

\subsection{Acknowledgements}
The authors are indebted to several people who generously provided guidance on the history of Katznelson's Question in its various forms, historical references, and the current standing of open problems: Vitaly Bergelson, Eli Glasner, John Griesmer, Karl Petersen, and Imre Ruzsa. The authors also thank the anonymous referee, whose shrewd feedback led to numerous clarifications, \cref{rmk_special_case_of_seq_form}, and a strengthening of Theorem C. Finally, thanks goes to Felipe Hern\'{a}ndez, who pointed out a mistake in an earlier version of \cref{def_winding_number}. The third author is supported by the National Science Foundation under grant number DMS~1901453.

\section{Notation, terminology, and prerequisites}
\label{sec_dynamical_prerecs}

We denote the set of integers and positive integers by $\Z$ and $\N$, respectively.  The (additive) 1-torus $\R / \Z$ is denoted by $\T$ and equipped with the metric induced by the function $\| \cdot \|: \R \to [0,1/2]$ that measures the Euclidean distance to the nearest integer.  Throughout, for convenience, Cartesian products of metric spaces are equipped with the $L^1$ (taxicab) metric.

\subsection{Combinatorics and topological dynamics}
For $A, B \subseteq \Z$ and $n \in \Z$, define
\begin{align*}
    A-n &= \big\{ m \in \Z \ \big| \ n + m  \in A \big\}, & nA &= \big\{n m \in \Z \ \big| \ m \in A \big\},\\
    A-B &= \big\{ a-b \ \big| \ a \in A, \ b \in B \big\}, &  A/n &= \big\{m \in \Z \ \big| \ nm \in A \big\}.
\end{align*}

As defined in the introduction, a (topological dynamical) \emph{system} $(X,T)$ is a pair consisting of a compact metric space $(X,d_X)$ and a continuous map $T: X \to X$.  A system $(X,T)$ is \emph{minimal} if for all $x \in X$, the set $\{T^n x \ | \ n \in \N\}$ is dense in $X$.

In this paper, we will focus on the recurrence of points in systems.  The following definition helps to make this precise. (The interested reader can consult \cite[Theorem 2.3]{HostKraMaass2016} and \cite[Theorems~5.3 and 5.6]{boshernitzanglasner2009} for a number of other equivalent characterizations of sets of topological recurrence, including the one mentioned in the first paragraph of \cref{sec_overview}.)

\begin{definition}\leavevmode
\label{def_eps_returns_and_recurrence}
The \emph{set of $\eps$-returns} of a system $(X,T)$ is
\[\dyret_\eps(X,T) = \big\{ m \in \N \ \big| \ \inf_{x \in X} d_X(x, T^m x) < \eps \big\}.\]
A set $R \subseteq \N$ is a \emph{set of topological recurrence} if for all systems $(X,T)$ and all $\eps > 0$, 
\[R \cap \dyret_\eps (X, T) \neq \emptyset.\]
\end{definition}

\begin{lemma}
\label{corollary_sufficient_to_prove_BR_property_for_power_of_system}
Let $(X,T)$ be a system.  For all $k \in \N$ and $\eps > 0$,
\[\dyret_\eps (X, T^k) = \dyret_\eps(X,T)/k.\]
\end{lemma}

\begin{proof}
Let $\eps > 0$.  The conclusion of the lemma follows by noting that for all $m \in \N$, both of the conditions $m \in \dyret_\eps (X, T^k)$ and $m \in \dyret_\eps (X, T) / k$ are equivalent to the existence of $x \in X$ such that $d_X(x, T^{mk} x) < \eps$.
\end{proof}

\subsection{Bohr sets, almost periodicity, and equicontinuity}

Bohr sets, which play a central role in Katznelson's Question and in this paper, are closely related to the topics of almost periodicity and equicontinuity.  In this section, we define \bz{} sets and collect the prerequisite results necessary for the proofs of main theorems.

\begin{definition}
\label{def_bohr_set}
A set $A \subseteq \N$ is a \emph{\bz{} set} if it contains the positive elements of a Bohr neighborhood of zero, that is, if there exists $\delta > 0$, $d \in \N$, and $\alpha \in \T^d$ such that
    \[\big\{ n \in \N \ \big| \ \| n \alpha \| < \delta \big\} \subseteq A.\]
The set $A$ is a \emph{Bohr set} if there exists $n \in \N$ such that $A-n$ is a \bz{} set. A subset of $\N$ is a \emph{\bzs{} set} (also, a \emph{set of Bohr recurrence}) if it has non-empty intersection with all \bz{} subsets of $\N$.
\end{definition}

\begin{remark}
\label{rmk_equiv_forms_of_bohr_and_bohrzero}
The family of \bz{} subsets of $\N$ is a filter: it is upward closed and closed under intersections.  Dually, the family of \bzs{} subsets of $\N$ is partition regular: at least one cell of any finite partition of a \bzs{} set is a \bzs{} set.  A set is \bz{} if and only if it has non-empty intersection with all \bzs{} sets.\footnote{That every \bz{} set has non-empty intersection with every \bzs{} set follows by definition. Conversely, suppose that $A$ has the property that $A\cap B\neq \emptyset$ for all \bzs{} sets $B.$ It follows that $\N\setminus A$ is not a \bzs{} set. By the definition of \bzs{} sets, there exists a \bz{} set $B_0$ such that $(\N\setminus A)\cap B_0=\emptyset$, whereby $B_0\subseteq A.$ Supersets of \bz{} sets are \bz{} sets, so $A$ is a \bz{} set.}  Also, note that $B \subseteq \N$ is a \bzs{} set if and only if for all $d \in \N$ and $\alpha \in \T^d$, $\inf_{n \in B} \| n \alpha \| = 0$.  This helps to explain why such sets are called sets of Bohr recurrence; see also the terminology in \cref{def_eps_returns_and_recurrence}.
\end{remark}

\begin{remark}
\label{rmk_Bohr_topology}
The completion of $\Z$ with the Bohr topology -- the topology generated by Bohr neighborhoods of zero and their translates -- yields its \emph{Bohr compactification}, $b\Z$.  Addition on $\Z$ induces a binary operation on $b\Z$ that makes it a compact (non-metrizable) abelian group.  In this context, a set $A \subseteq \N$ is a \bz{} set if and only if it contains the preimage (under the canonical injection of $\N$ into $b\Z$) of an open neighborhood of $0$ in $b\Z$, and a set $B \subseteq \N$ is a \bzs{} set if and only if $0$ is an accumulation point of the image of $B$ in $b\Z$.  We mention the Bohr compactification here only to help motivate the terminology; we do not have any use for particulars concerning $b\Z$ in this paper, so we do not develop the details any further.
\end{remark}

\begin{lemma}
\label{lemma_dilates_of_bohr_are_bohr}
If $B \subseteq \N$ is a \bz{} set, then for all $m \in \N$, the sets $mB$ and $B / m$ are \bz{} sets.
\end{lemma}

\begin{proof}
Let $\eps > 0$ and $\alpha \in \T^d$ be such that
\[C:=\big\{ n \in \N \ \big| \ \|n \alpha \| < \eps \big\} \subseteq B.\] 
For $m\in \N,$ let
\[D:=\left\{n \in \N \ \big| \ \left\|n \frac{\alpha}{m} \right\| < \eps \right\} \text{ and } E:=\left\{n \in \N \ \big| \ \left\|n m\alpha \right\| < \eps \right\}.\]
It is easy to check that
\[D/m\subseteq C\;\;\text{and}\;\;E\subseteq C/m.\]
The result follows by the relation $D\cap m\N\subseteq mC\subseteq mB,$ as both $D$ and $m\N$ are \bz{} sets, and by the fact that $E\subseteq C/m\subseteq B/m$.
\end{proof}

Let $G$ be a compact abelian group and $T: G \to G$ be addition by a fixed element $g \in G$.  The map $T$ is an isometry (with respect to a translation-invariant metric $d_G$ on $G$), and the set of times at which a point $x \in G$ visits a non-empty open set $U \subseteq G$ is a Bohr set.  In fact, the same conclusion can be reached under the weaker, topological assumption that the family of maps $\{T^n \ | \ n \in \N\}$ is equicontinuous; see \cref{lemma_equicontinuous_return_times} below and the remark following it.

\begin{definition}
\label{def_equicontinuous_system}
A system $(X,T)$ is \emph{equicontinuous} if the family of maps $\{T^n \ | \ n \in \N\}$ is equicontinuous, i.e., for all $\eps > 0$, there exists $\delta > 0$ such that for all $x, y \in X$ with $d_X(x,y) < \delta$ and all $n \in \N$, $d_X(T^n x, T^n y) < \eps$.
\end{definition}

Equicontinuity is closely related to the dynamical phenomenon of almost periodicity, defined next.  See \cref{lemma_equicontinuous_observables_are_ap} below for the connection which is most relevant to this work.

\begin{definition}
\label{def_bohrap_sequence}
Let $(X,d_X)$ be a metric space, and let $f: \N \to X$. The sequence $f$ is \emph{(Bohr) almost periodic on $\N$} if for all $\eps > 0$, the \emph{set of $\eps$-almost periods}
    \[\big\{ m \in \N \ \big| \ \sup_{n \in \N} d_X \big( f(n + m), f(n) \big) < \eps \big\}\]
is a \bz{} set. Replacing all instances of $\N$ with $\Z$ yields the definition of a Bohr almost periodic function on $\Z$ as defined in the introduction. The \emph{mean} of a real-valued almost periodic sequence $f: \N \to \R$ is the quantity $\lim_{N \to \infty} N^{-1} \sum_{n=1}^N f(n)$.
\end{definition}

The following is a collection of useful classical results relating Bohr sets, almost periodicity, and equicontinuity; see \cite[Chapter 4]{petersen_ergodic_theory} for a modern presentation of the ideas.

\begin{lemma}
\label{lemma_equicontinuous_observables_are_ap}
Let $f: \N \to \R$. The following are equivalent:
\begin{enumerate}
    \item the sequence $f$ is almost periodic;
    \item there exists an equicontinuous system $(X,T)$, a point $x \in X$, and a continuous function $h: X \to \R$ such that for all $n \in \N$, $f(n) = h(T^n x)$.
\end{enumerate}
Moreover, the same statement holds with $\R$ replaced by $\T$ and with ``equicontinuous system'' replaced by ``minimal equicontinuous system'' in condition (2).
\end{lemma}

It is a well-known consequence of equicontinuity, at least in minimal systems, that the set of return times of a point to a neighborhood of itself is a \bz{} set, but we were unable to find a convenient reference in the literature concerning non-minimal systems.  The argument is short so we provide it here.

\begin{lemma}
\label{lemma_equicontinuous_return_times}
Let $(X,T)$ be an equicontinuous system.  For all $\eps > 0$, the set
\begin{align}
\label{eqn_all_points_return_times_equicontinuous}
    \{n \in \N \ | \ \sup_{x \in X} d_X(x, T^n x) < \eps \}
\end{align}
is a \bz{} set.
\end{lemma}

\begin{proof}
First we will show that for all $x \in X$ and $\eps > 0$, the set
\begin{align}
\label{eqn_point_return_times_set}
    A_{x, \eps} \defeq \big\{n \in \N \ \big| \ d_X(x,T^n x) < \eps \big\}
\end{align}
is a \bz{} set.  Let $h: X \to [0,1]$ be continuous, equal to 1 at $x$, and equal to 0 outside of an open ball of radius $\eps$ about $x$.  It follows from  \cref{lemma_equicontinuous_observables_are_ap} that the sequence $f: n \mapsto h(T^n x)$ is almost periodic.  Since $h(x) = 1$, the set of $(1/2)$-almost periods of $f$, a \bz{} set, is contained in $A_{x, \eps}$.

Now we will prove the statement in the lemma.  Let $\eps > 0$.  Let $0<\delta < \eps / 3$ be sufficiently small so that for all $x, y \in X$ with $d_X(x,y) < \delta$ and for all $n \in \N$, $d_X(T^n x, T^n y) < \eps / 3$.  Let $Y$ be a $\delta$-dense subset of $X$.  By the previous paragraph, the set
\[A_\eps \defeq \bigcap_{y \in Y} A_{y, \eps / 3},\]
where $A_{y, \eps / 3}$ is defined as in \eqref{eqn_point_return_times_set}, is a \bz{} set since it is the intersection of finitely many \bz{} sets.

We will show that $A_\eps$ is a subset of the set in \eqref{eqn_all_points_return_times_equicontinuous}. Let $n \in A_\eps$ and $x \in X$. There exists $y \in Y$ such that $d_X(x,y) < \delta$. Since $d_X(T^n x, T^n y) < \eps/3$, $d_X(y, T^n y) < \eps / 3$, and $\delta < \eps / 3$, we have by the triangle inequality that $d_X(x, T^n x) < \eps$, as was to be shown.
\end{proof}

In fact, \bz{} sets can be used to characterize equicontinuous systems: a minimal system $(X,T)$ is equicontinuous if and only if for all $x \in X$ and $\eps > 0$, the set $\{n \in \N \ | \ d_X( x, T^n x) < \eps \}$ is a \bz{} set.  We do not have need for this fact, so we omit the proof.

\subsection{Dynamical forms of Katznelson's Question}
\label{sec_dynamically_equivalent_forms}

Katznelson's Question can be stated in several different equivalent forms. In this section, we describe two dynamical forms; some of its combinatorial forms are presented in \cref{sec_comb_forms_of_katznelson}.  For our purposes, it will be most convenient to phrase Katznelson's Question in terms of the size of the set of $\eps$-returns of a system.

\begin{definition}
\label{def_bz_large_returns}
A system $(X,T)$ has \emph{\bz{} large returns} if for all $\eps > 0$, the set of $\eps$-returns $\dyret_\eps(X,T)$ is a \bz{} set.
\end{definition}

Katznelson's Question and the following one are equivalent, in the sense that one has a positive answer if and only if the other does.  It is this formulation of Katznelson's Question that we will address in the next section.

\begin{namedthm}{Question D1}
    Do all topological dynamical systems have \bz{} large returns?
\end{namedthm}

That Katznelson's Question and Question D1 are equivalent follows from the marginally finer fact that a system $(X,T)$ has \bz{} large returns if and only if sets of Bohr recurrence are sets of recurrence for $(X,T)$.  Indeed, suppose $(X,T)$ has \bz{} large returns, and let $R \subseteq \N$ be a set of Bohr recurrence.  As explained in \cref{rmk_equiv_forms_of_bohr_and_bohrzero}, the set $R$ is a \bzs{} set.  For all $\eps > 0$, the set $\dyret_\eps(X,T)$ is a \bz{} set, whereby $\dyret_\eps(X,T) \cap R \neq \emptyset$.  Since $\eps > 0$ was arbitrary, the set $R$ is a set of recurrence for $(X,T)$.  Conversely, suppose that a system $(X,T)$ does not have \bz{} large returns: there exists $\eps > 0$ such that $\dyret_\eps(X,T)$ is not a \bz{} set.  As explained in \cref{rmk_equiv_forms_of_bohr_and_bohrzero}, the set $R \defeq \N \setminus \dyret_\eps(X,T)$ is a \bzs{} set, a set of Bohr recurrence, that is not a set of recurrence for $(X,T)$.

\begin{namedthm}{Question D2}
    If $(X,T)$ is a minimal topological dynamical system, is it true that for all non-empty, open $U \subseteq X$, the set
    \[\big\{ n \in \N \ \big| \ U \cap T^{-n}U \neq \emptyset \big\}\]
    is a \bz{} set?
\end{namedthm}

Questions D1 and D2 are equivalent.  Indeed, that a positive answer to Question D2 implies one for D1 follows from the fact that any system $(X,T)$ contains a minimal subsystem $(X',T)$, and $\dyret_\eps(X',T) \subseteq \dyret_\eps(X,T)$.  On the other hand, a positive answer to Question D1 combines with the following lemma to immediately give a positive answer to Question D2.

\begin{lemma}
\label{lemma_containment_amongst_recurrence_types}
Let $(X,T)$ be a minimal system.  For all non-empty, open $U \subseteq X$, there exists $\eps > 0$ such that
\[\dyret_{\eps} (X,T) \subseteq \big\{ n \in \N \ \big| \ U \cap T^{-n}U \neq \emptyset \big\} \subseteq \dyret_{\diam(U)} (X,T).\]
\end{lemma}

\begin{proof} The second containment is immediate: if $x \in U \cap T^{-m}U \neq \emptyset$, then $d_X(x,T^m x) \allowbreak < \diam(U)$, whereby $m \in \dyret_{\diam(U)} (X,T)$. To see the first, let $\delta > 0$ be such that $U$ contains a non-empty, open set $U'$ and its $\delta$-neighborhood.
Since $(X,T)$ is minimal, there exists $N \in \N$ such that for all $x \in X$, there exists $n \leq N$ such that $T^n x \in U'$.  Let $\eps > 0$ be such that for all $x, y \in X$ with $d_X(x,y) < \eps$ and for all $n \leq N$, $d_X(T^n x, T^n y) < \delta$.  Now, if $m \in \dyret_\eps(X,T)$, there exists $x \in X$ such that $d_X(x,T^mx) < \eps$.  It follows that there exists $n \leq N$ such that $d_X(T^n x, T^{n+m} x) < \delta$ and $T^n x \in U'$.  Therefore, $T^n x, T^{n+m} x \in U$, whereby $U \cap T^{-m}U \neq \emptyset$.
\end{proof}

\section{Recurrence and hidden frequencies in skew product systems}
\label{sec_recurrence_in_skew_towers}

In this section, we prove Theorems \ref{maintheorem_katznelson_for_skew_towers} and \ref{maintheorem_hidden_frequencies}.  The first gives a positive answer to Katznelson's Question for certain towers of skew product extensions by $1$-tori over equicontinuous systems, while the second demonstrates that skew product extensions can introduce new ``frequencies'' that must be controlled to ensure recurrence.

\subsection{Skew product dynamical systems}

We collect here the basic notation and terminology for skew product systems, winding numbers, and lifts of torus-valued maps.

\begin{definition}
\label{def_skew_prod_systems}
Let $(X,T)$ be a system and $h: X \to \T$ be a continuous map. The \emph{skew product system $(X \times \T,T_h)$} is defined by $T_h: X \times \T \to X \times \T$ where
\[T_h(x,t) = \big(Tx, t + h(x) \big).\]
For $m \in \N \cup \{0\}$, define $h_m: X \to \T$ by $h_0 \equiv 0$ and
\[h_m(x) = \sum_{i=0}^{m-1} h(T^i x),\]
so that $T_h^m (x,t) = \big( T^m x, t+ h_m(x) \big)$.
\end{definition}

We will frequently consider real-valued skewing functions $H: X \to \R$; the skew product system $(X \times \T, T_H)$ in this case is defined by implicitly composing the map $H$ with the quotient map $\pi: \R \to \T$.

\begin{definition}
\label{def_winding_number}
Let $h: \T \to \T$ be continuous.  There exists a continuous map $\varphi: [0,1] \to \R$ with the property that $\pi \circ \varphi = h \circ \big(\restr{\pi}{[0,1]} \big)$. The \emph{winding number of $h$} is equal to $\varphi(1) - \varphi(0)$; it is an integer that can be shown to be independent of $\varphi$. 
If $h$ has winding number equal to zero then $\varphi$ descends to a continuous function $H: \T \to \R$ satisfying $\pi \circ H = h$. We refer to $H$ as the \emph{continuous lift of $h$ to $\R$}.
\end{definition}

\begin{remark}
The winding number of a continuous function $h: \T \to \T$ counts the number of times the function $h$ ``wraps around'' the circle. The winding number of a sum of functions is the sum of their winding numbers.  For $\alpha \in \T$, the winding number of $x \mapsto h(x+\alpha)$ is easily seen to be equal to the winding number of $h$. It follows that the winding number of the function $h_m$, defined in \cref{def_skew_prod_systems}, is $m$ times the winding number of $h$.
\end{remark}

\subsection{Returns in towers over equicontinuous systems: a proof of \cref{maintheorem_katznelson_for_skew_towers}}
\label{sec_returns_in_towers}

In this section, we prove \cref{maintheorem_katznelson_for_skew_towers} using the reformulation of Katznelson's Question described in \cref{sec_dynamically_equivalent_forms}.  At the heart of \cref{maintheorem_katznelson_for_skew_towers} is a simple idea that is quickly illustrated in the case of a single skew product by the 1-torus over a rotation on the 1-torus.

\begin{namedthm}{Special case of \cref{maintheorem_katznelson_for_skew_towers}}
For all $\alpha \in \T$ and all continuous $h: \T \to \T$, the skew product system $(\T^2, T_h)$,
\[T_h(x,t) = \big(x + \alpha, t+ h(x)\big),\]
has \bz{} large returns.
\end{namedthm}

\begin{proof}
Let $\alpha \in \T$ and $h: \T \to \T$ be continuous.  Let $\eps > 0$.  In order to show that $m \in \dyret_\eps(\T^2, T_h)$, we must demonstrate the existence of a point $(x,t) \in \T^2$ for which $\|m \alpha\| < \eps$ and $\|h_m(x)\| < \eps$. (Recall that all Cartesian products in this work are equipped with the $L^1$ metric.)

Let $m \in \N$. If $h$ has non-zero winding number, then so does $h_m$, and it follows by the intermediate value theorem that there exists $x \in \T$ such that $h_m(x) = 0$. It follows that $\{m \in \N \ | \ \|m \alpha \| < \eps\} \subseteq \dyret_\eps(\T^2, T_h)$.

If, on the other hand, the function $h$ has zero winding number, then it has a continuous lift $H: \T \to \R$.  Put $\beta = \int_\T H(x) \ dx$.  For any $m \in \N$, the mean value theorem for integrals gives the existence of a point $x \in \T$ such that $H_m(x) = m \beta$.   This implies that $\{m \in \N \ | \ \|m (\alpha, \beta) \| < \eps\} \subseteq \dyret_\eps(\T^2, T_h)$.

In either case, we find that $\dyret_\eps(\T^2, T_h)$ contains a \bz{} set, whereby $(\T^2, T_h)$ has \bz{} large returns.
\end{proof}

To prove \cref{maintheorem_katznelson_for_skew_towers}, we improve on this idea in two ways.  First, we replace the base $1$-toral rotation by a general equicontinuous system.  If $X$ is totally disconnected -- as it is when $(X,T)$ is an odometer, for example -- the argument can no longer appeal to winding numbers or the intermediate value theorem.  The fact that the result continues to hold for not-necessarily-connected base systems shows that the it has less to do with connectedness and, as we will see, more to do with the fact that the real numbers are well-ordered.  Second, to extend the result to certain towers of skew product extensions, we iterate the argument, using the fact that entire fibers $\{x\} \times \T$ exhibit recurrence.

The first step in the proof of \cref{maintheorem_katznelson_for_skew_towers} is to show that partial sums of real-valued, almost periodic sequences are close to their mean along a \bz{} set.

\begin{proposition}
\label{prop_ap_sequences_near_mean_along_bz_set}
Let $f: \N \to \R$ be almost periodic, and let $\beta \in \R$ be its mean.  For all $\eps > 0$, there exists a \bz{} set $B \subseteq \N$ such that for all $m \in B$, there exists $n \in \N$ such that
\begin{align}
\label{eqn_skew_sum_is_close_to_mean}
    \left| \sum_{i = 0}^{m-1} f(n + i) - m \beta \right| < \eps.
\end{align}
\end{proposition}

\begin{proof}
Let $\eps > 0$. Let $B \subseteq \N$ be the set of $\eps / 2$-almost periods for $f$. The set $B$ is a \bz{} set that we will show satisfies the conclusions of the proposition.

Let $m \in B$. Define $f_m: \N \to \R$ by
\[f_m(n) = \sum_{i=0}^{m-1} f(n + i),\]
and note that $f_m$ has mean $m \beta$.  Because $m$ is an $\eps / 2$-almost period for $f$, the sequence $f_m$ takes ``$\eps$-steps,'' in the sense that for all $n \in \N$, $|f_m(n+1) - f_m(n)| < \eps.$  Since $f_m$ has mean $m \beta$ and it takes $\eps$-steps, there exists $n \in \N$ for which $|f_m(n) - m \beta| < \eps$, as was to be shown.
\end{proof}

The following theorem proves \cref{maintheorem_katznelson_for_skew_towers} in the case of a single skew product extension over an equicontinuous system.

\begin{theorem}
\label{theorem_skew_extensions_of_equicontinuous_have_kv_property}
Let $(X,T)$ be an equicontinuous system, and let $h: X \to \T$ be continuous. The skew product system $(X \times \T, T_h)$ has \bz{} large returns.
\end{theorem}

\begin{proof}
Because $(X,T)$ is equicontinuous, \cref{lemma_equicontinuous_return_times} gives that the set
\begin{align}
\label{eqn_def_of_set_of_returns_in_proof_of_first_thm_a}
    B \defeq \big\{ n \in \N \ \big| \ \sup_{x \in X} d_X(x, T^n x) < \eps / 2 \big\}
\end{align}
is a \bz{} set.

We consider two cases.  In Case 1, for all $m \in \N$, the point 0 is in the image of the map $h_m$, while in Case 2, there exists $m \in \N$ for which 0 is not in the image of the map $h_m$.

Suppose we are in Case 1.  To see that $(X,T)$ has \bz{} large returns, let $\eps > 0$.   We claim that $B \subseteq \dyret_\eps(X\times \T,T_h)$. Let $m \in B$. Since 0 is in the image of the map $h_m$, there exists $x \in X$ such that $h_m(x) = 0$.  It follows that $T_h^m (x,0) = (T^m x, 0)$, whereby $m \in \dyret_\eps(X\times \T,T_h)$, as was to be shown.

Suppose we are in Case 2 so that there exist $m_0 \in \N$ for which $0\notin h_{m_0}(X)$. We claim that we can assume that $m_0 = 1$. Indeed, to prove that $(X,T)$ has the \bz{} large returns, it suffices by \cref{corollary_sufficient_to_prove_BR_property_for_power_of_system} to prove that the system $(X\times \T,T_h^{m_0})$ has \bz{} large returns.  Define $S = T^{m_0}$. Note that $T_h^{m_0} = S_{h_{m_0}}$ (following the notation established in \cref{def_skew_prod_systems}), so that $(X\times \T,T_h^{m_0}) = (X \times \T, S_{h_{m_0}})$.  Since $(X, S)$ is equicontinuous and $h_{m_0}: X \to \T$ is continuous, we can proceed by replacing $S$ by $T$, $h_{m_0}$ by $h$, and under the assumption that $0$ is not in the image of the map $h: X \to \T$.

Let $\pi: \R \to \T$ be the quotient map from $\R$ to $\T$, let $\rotpi: \T \to \R$ be a section of $\pi$ that is continuous at all points of $\T$ except 0, and define $H = \rotpi \circ h$. Since $0$ is not in the image of $h$, the map $H\colon X \to \R$ is continuous. Moreover, $\pi \circ H = h$ by construction.
Fix $x_0 \in X$ and define $f: \N \to \R$ by $f(n) = H(T^{n} x_0)$. The system $(X, T)$ is equicontinuous, so by \cref{lemma_equicontinuous_observables_are_ap} the sequence $f$ is almost periodic.

Let $\eps > 0$; our aim is to show that $\dyret_\eps (X \times \T, T_h)$ is a \bz{} set.  Let $\beta$ be the mean of $f$, and let $B' \subseteq \N$ be the \bz{} set from \cref{prop_ap_sequences_near_mean_along_bz_set} (with $\eps / 4$ as $\eps$).
Define
\[B'' = B \cap B' \cap \big\{ m \in \N \ \big| \ \|m \beta \| < \eps / 4 \big\}.\]
Note that $B''$ is a \bz{} set.  We will show that $B'' \subseteq \dyret_\eps (X \times \T, T_h)$.

Let $m \in B''$. It follows by \cref{prop_ap_sequences_near_mean_along_bz_set} and the fact that $\|m \beta \| < \eps / 4$ that there exists $n \in \N$ such that
\begin{align}
\label{eqn_skew_sum_is_close_to_zero_on_torus_in_proof}
    \left\| \sum_{i = 0}^{m-1} f(n + i) \right\| < \eps / 2.
\end{align}
Define $x = T^{n} x_0$.  Note that $f(n+i) = H(T^{i} x)$. It follows by the definition of $H$ and \eqref{eqn_skew_sum_is_close_to_zero_on_torus_in_proof} that
\[\big\|h_m(x) \big\| < \eps / 2.\]
Since $m \in B$, we have additionally that $d_X(T^{m} x, x) < \eps / 2$.  Combining these facts, we see that
\begin{align*}
    d_{X\times \T} \big (T_h^{m} (x,0), (x,0) \big) &= d_{X\times \T} \big( (T^{m} x, h_m(x)), (x, 0) \big) \\
    &= d_{X} \big( T^{m} x, x \big)  + \| h_m(x) \| < \eps.
\end{align*}
It follows that $m \in \dyret_\eps (X \times \T, T_h)$, as was to be shown.
\end{proof}

In the following theorem, we establish the inductive step for an iterative procedure that allows us to handle the multiple skew product extensions that appear in \cref{maintheorem_katznelson_for_skew_towers}.

\begin{theorem}
\label{theorem_dynamical_inductive_step}
Let $(X,T)$ be a system and $h: X \to \T$ be continuous.  If the skew product system $(X \times \T, T_h)$ has \bz{} large returns, then for all continuous $g: \T \to \T$, the skew product system $(X \times \T^2, T_{h,g}$) defined by
\[T_{h,g} (x,t,s) = \big(T_h(x,t), s+g(t) \big) = \big( Tx, t + h(x), s + g(t) \big)\]
has \bz{} large returns. 
\end{theorem}

\begin{proof}
Let $g: \T \to \T$ be continuous.  If $g$ has non-zero winding number, put $\gamma = 0$.  If $g$ has winding number equal to zero, let $G: \T \to \R$ be a continuous lift of $g$ to $\R$, and put $\gamma = \int_{\T} G(t) \ dt$.

Let $\eps > 0$.  It follows from our assumptions that the set
\[B \defeq \dyret_\eps(X \times \T, T_h) \cap \big\{n \in \N \ \big| \ \| n \gamma \big\| < \eps \big\}\]
is a \bz{} set.  We will show that $B \subseteq \dyret_\eps(X \times \T^2, T_{h,g})$.

Let $m \in B$. To show that $m \in \dyret_\eps(X \times \T^2, T_{h,g})$, we will show that there exists $(x, t, s) \in X \times \T^2$ such that
\begin{align}
\label{eqn_goal_to_show_recurrence_in_double_skew}
    d_{X \times \T^2}\big( (x,t,s), T_{h,g}^m (x,t,s) \big) < \eps.
\end{align}
Note that
\[T_{h,g}^m(x,t,s)=\big(T^m x, t+h_m(x), s+ g_{x;m}(t) \big),\]
where the third coordinate function, denoted by $g_{x;m}: \T \to \T$, depends on $x$ and is defined by
\begin{align}
\label{eqn_def_of_gmx}
    g_{x;m}(t) = \sum_{i=0}^{m-1}g\big(t+h_{i}(x)\big).
\end{align}
For all $x \in \T$, the winding number of $g_{x;m}$ is $m$ times the winding number of $g$, and, in the case that $g$ has winding number equal to zero, the function $G_{x;m}: \T \to \R$, defined by replacing $g$ with $G$ in \eqref{eqn_def_of_gmx}, is a continuous lift of $g_{x;m}$ to $\R$ with $\int_{\T}G_{x;m} (t) \ dt = m \gamma$.

Since $m \in \dyret_\eps(X \times \T, T_h)$, there exists $(x,t) \in X \times \T$ such that
\begin{align}
\label{eqn_point_return_in_base_skew_product}
    d_{X \times \T} \big( (x,t), T_h^m (x,t) \big) < \eps.
\end{align}
Since $T_h$ commutes with rotation in the second coordinate, it follows, in fact, that \eqref{eqn_point_return_in_base_skew_product} holds for all $t \in \T$.

If $g$ has non-zero winding number, then so does $g_{x;m}$; it follows by the intermediate value theorem that there exists $t \in \T$ such that $g_{x;m}(t) = 0$.  On the other hand, if $g$ has winding number equal to zero, then the mean value theorem for integrals combines with the fact that $G_{x;m}$ is a continuous lift of $g_{x;m}$ to $\R$ with mean $m \gamma$ to guarantee the existence of a $t \in \T$ such that $G_{x;m}(t) = m \gamma$.  Since $m \in B$, it follows that $\| g_{x;m}(t)\| < \eps$.

In either case, it follows from \eqref{eqn_point_return_in_base_skew_product} that for all $s \in \T$, the point $(x,t,s)$ satisfies \eqref{eqn_goal_to_show_recurrence_in_double_skew}, as was to be shown.
\end{proof}

Combining Theorems \ref{theorem_skew_extensions_of_equicontinuous_have_kv_property} and \ref{theorem_dynamical_inductive_step}, we can prove \cref{maintheorem_katznelson_for_skew_towers}.

\begin{proof}[Proof of \cref{maintheorem_katznelson_for_skew_towers}]
According to the reformulation of Katznelson's Question in Question D1 in \cref{sec_dynamically_equivalent_forms}, we need to prove that the systems described in the statement of \cref{maintheorem_katznelson_for_skew_towers} have \bz{} large returns.  That fact follows by a simple induction argument, appealing to \cref{theorem_skew_extensions_of_equicontinuous_have_kv_property} for the base case and \cref{theorem_dynamical_inductive_step} for the inductive step.
\end{proof}

\begin{remark}
\label{remark_details_on_hidden_freq}
The proofs of Theorems \ref{theorem_skew_extensions_of_equicontinuous_have_kv_property} and \ref{theorem_dynamical_inductive_step} tell us which frequencies it suffices to control to ensure recurrence in a tower of skew product extensions of the type described in \cref{maintheorem_katznelson_for_skew_towers}. For those skewing functions $h_i$ with zero winding number, we must control for the average $\int_\T H_i(t) \ dt$ of a continuous lift of $h_i$; those skewing functions with non-zero winding number do not introduce any additional frequencies. The example in \cref{maintheorem_hidden_frequencies} shows that controlling for the averages of the skewing functions with zero winding number is indeed necessary for recurrence.  In the case of more general isometric extensions of equicontinuous systems, we do not know how to identify frequencies beyond those in the equicontinuous factor that influence recurrence.
\end{remark}

\begin{corollary}
Sets of Bohr recurrence are sets of recurrence for factors, proximal extensions (cf. footnote \ref{footnote_proximal_definition}), and inverse limits of the types of skew product tower systems described in the statement of \cref{maintheorem_katznelson_for_skew_towers}.
\end{corollary}

\begin{proof}
It is easy to check that factors of systems and inverse limits of families of systems that have \bz{} large returns also have \bz{} large returns.  It is a consequence of \cite[Proposition 3.8]{HostKraMaass2016} that proximal extensions of systems with \bz{} large returns have \bz{} large returns. Thus, the statement in question is an immediate corollary of \cref{maintheorem_katznelson_for_skew_towers} and the equivalences between the different forms of Katznelson's Question described in \cref{sec_dynamically_equivalent_forms}.
\end{proof}

\subsection{Skew product extensions by the 1-torus}

The two main results in this section concern skew-product extensions of equicontinuous systems by the 1-torus. \cref{theorem_not_minimal_implies_equicontinuous} concerns general equicontinuous systems; the result will be useful in the proof of Theorem C but may also be of independent interest.  \cref{theorem_condition_for_beta_skew_to_be_minimal} is a complement to the classic theorem of Gottschalk and Hedlund; we elaborate on this in \cref{rmk_gottschalk_hedlund_connection} after the proof.

\begin{theorem}
\label{theorem_not_minimal_implies_equicontinuous}
Let $(X, T)$ be a minimal, equicontinuous system, and let $h: X \to \T$ be continuous.  If the skew-product system $(X \times \T, T_h)$ is not minimal, then it is equicontinuous.
\end{theorem}

\begin{proof}
Fix $x_0 \in X$, and let $Z = \overline{\{T_h^n (x_0,0) \ | \ n \in \N \}}$.  Making use of the family of automorphisms of $(X \times \T, T_h)$ described by $(x,t) \mapsto (x,t+s)$, $s \in \T$, there exists a closed subgroup $F$ of $\T$ such that $Z \cap (\{x_0\} \times \T) = \{x_0\} \times F$; the details are left to the reader as an exercise.  Since $(X,T)$ is minimal, the system $(X \times \T, T_h)$ is minimal if and only if $F = \T$.

Suppose that the system $(X \times \T, T_h)$ is not minimal so that $F \neq \T$; there exists $\ell \in \N$ such that $F = \{0, \ldots, \ell-1\} / \ell$.  We will show that the system $(X \times \T, T_h)$ is equicontinuous.  Using the family of second-coordinate rotation automorphisms, it suffices to show that the system $(Z, T_h)$ is equicontinuous.

Let $K = \T / F$; it is a 1-torus with metric induced by $\| \cdot \|_K$, the Euclidean distance to zero. Let $\pi_K: \T \to K$ be the quotient map, and consider the skew product system $(X \times K, T_{\pi_K \circ h})$, where we endow $X \times K$ with the $L^1$ metric.  Let $Y = \overline{\{T_{\pi_K \circ h}^n (x_0,0) \ | \ n \in \N \}}$. Following the same reasoning as above, by the definition of $F$, for all $x \in X$, $|Y \cap (\{x\} \times K)| = 1$.

Let $g: X \to K$ be such that for all $x \in X$, $Y \cap (\{x\} \times K) = \{\big(x,g(x) \big)\}$; that is, the graph of $g$ is equal to $Y$. It follows that
\[Z = \bigcup_{x \in X} \Big(\big\{x\} \times \pi_K^{-1} \big\{g(x) \big\} \Big).\]
Since $Y$ is closed, the map $g$ is continuous.  Note that in the system $(X \times K, T_{\pi_K \circ h})$, we have for all $x \in X$ that $T_{\pi_K \circ h} (x,g(x)) = (Tx,g(x) + \pi_K(h(x))) = (Tx, g(Tx))$, whereby $\pi_K \circ h_n = g\circ T^n - g$ for all $n\in\N\cup\{0\}$. 

We will leverage the fact that $F$ is discrete to show that the system $(Z, T_h)$ is equicontinuous. Recall that $|F|=\ell$.

\begin{claim}[Local lifts]
\label{claim_local_lifts}
There exists $\eta > 0$ such that for any ball $B \subseteq X$ of diameter at most $\eta$, there exists a continuous map $G: B \to \T$ such that $\pi_K \circ G = \restr{g}{B}$ (that is, $G$ is a continuous lift of $\restr{g}{B}$) and such that for all $x, x' \in B$, $\|G(x) - G(x') \| < 1/(4\ell)$.
\end{claim}

\begin{proof}
\renewcommand{\qedsymbol}{\rotatebox{45}{\#}}
Since $g: X \to K$ is continuous and $X$ is compact, there exists some $\eta>0$ such that for all $x,x'\in X$ with $d_X(x,x') \leq \eta$,
\begin{align}
\label{eqn_ucoX}
\|g(x)-g(x')\|_K\leq \frac{1}{8\ell}.
\end{align}
Let $B\subseteq X$ be any ball of diameter $\eta$ centered at a point $x_0\in X$. Let $w_0$ be any point in $\T$ for which $\pi_K(w_0)=g(x_0)$, and consider the interval
\[
I \defeq \left[w_0-\frac{1}{8\ell}, w_0+\frac{1}{8\ell}\right]\subseteq\T.
\] 
The projection map $\pi_K$ restricted to $I$, which we denote by $\varphi\colon I\to \pi_K(I)$, is continuous and surjective. Moreover, it is injective because the length of $I$ is smaller than $1/ (2\ell)$ and points that have the same image under $\pi_K$ are at least $1/\ell$ apart. Therefore the map $\varphi\colon I\to \pi_K(I)$ is a continuous bijection between compact spaces, which implies that it is a homeomorphism. 
Let $\varphi^{-1}\colon \pi_K(I)\to I$ denote its inverse.
Since $\pi_K(I)=\varphi(I)=\big[g(x)-1/(8\ell), g(x)+1/(8\ell) \big]$ and in light of \eqref{eqn_ucoX}, we have $g(B)\subseteq \varphi(I)$.
This ensures that the map $G=\varphi^{-1}\circ \restr{g}{B}$ is a well defined continuous function from $B$ to $\T$ satisfying $\|G(x) - G(x') \|\leq 1/(8\ell)< 1/(4\ell)$ for all $x, x' \in B$. Since we clearly have $\pi_K \circ G =\varphi\circ G= \varphi\circ \varphi^{-1}\circ \restr{g}{B}=\restr{g}{B}$, the proof is complete.
\end{proof}

Let $\{x_i\}_{i \in \mathcal{I}}$ be an $(\eta / 16)$-dense  subset of $X$.  For each $i \in \mathcal{I}$, define $B_i = B_{\eta / 2}(x_i)$ and appeal to \cref{claim_local_lifts} to find $G_i: B_i \to \T$, a continuous lift of $\restr{g}{B_i}$. Note that for all $i, j \in \mathcal{I}$ for which $B_i \cap B_j$ is non-empty, the function
\begin{align}
\label{eqn_consistency_of_atlas}
    \restr{G_i}{B_i \cap B_j} - \restr{G_j}{B_i \cap B_j} \text{ is a constant, $F$-valued function}.
\end{align}
Indeed, since both $G_i$ and $G_j$ are lifts of $g$, we have $\pi_K \circ \left(\restr{G_i}{B_i \cap B_j} - \restr{G_j}{B_i \cap B_j}\right) = 0$. Therefore, the function $\restr{G_i}{B_i \cap B_j} - \restr{G_j}{B_i \cap B_j}$ maps $B_i \cap B_j$ into $F$.  By \cref{claim_local_lifts}, the function
$\restr{G_i}{B_i \cap B_j} - \restr{G_j}{B_i \cap B_j}$ differs between two points on $B_i \cap B_j$ by at most $1/(2\ell)$.  Since any two points of $F$ are separated in distance by at least $1/\ell$, the conclusion follows.

Let $n \in \N$ be such that $\|T^n - \id\|_\infty$ $<\eta/16$; such an $n$ exists because $(X,T)$ is equicontinuous, and this $n$ will be fixed for the rest of the proof. For $i \in \mathcal{I}$, define $B_i' = B_{\eta / 4}(x_i)$, and note that $T^n B_i' \subseteq B_i$, so that $G_i \circ T^n$ is defined and is continuous on $B_i'$. 
Define $f_i: B_i' \to F$ by
\[f_i(x) = h_n(x) + G_i(x) - G_i(T^nx);\]
that $f_i$ takes values in $F$ can be seen by applying $\pi_K$ and using the fact that $\pi_K \circ h_n = g \circ T^n - g$.

Since $h_n$, $G_i$, and $G_i \circ T^n$ are all continuous on $B_i'$, the function $f_i$ is continuous. We claim that if $i, j \in \mathcal{I}$ are such that $B_i' \cap B_j' \neq \emptyset$, then $\restr{f_i}{B_i' \cap B_j'} = \restr{f_j}{B_i' \cap B_j'}$.
Indeed, for $x \in B_i' \cap B_j'$, there exists by \eqref{eqn_consistency_of_atlas} a value $c \in F$ such that $G_i(x) = G_j(x) + c$ and $G_i(T^n x) = G_j(T^n x)+c$.  It follows that
\begin{align}
\label{eqn_consistency_of_fis_to_get_f}
    \begin{aligned} f_i(x) &= h_n(x) + G_i(x) - G_i(T^nx) \\
    &= h_n(x) + \big(G_j(x) + c\big) - \big(G_j(T^nx) + c \big) \\
    &= h_n(x) + G_j(x) - G_j(T^nx) = f_j(x).\end{aligned}
\end{align}
Since $\{x_i\}_{i \in \mathcal{I}}$ is $(\eta / 16)$-dense, we have that $X = \cup_{i \in \mathcal{I}} B_i'$.  Therefore, by \eqref{eqn_consistency_of_fis_to_get_f}, we can define $f: X \to F$ by defining, for each $i \in \mathcal{I}$, the function $f$ to be equal to $f_i$ on $B_i'$.  Since each $f_i$ is continuous, the function $f$ is continuous.

Summarizing the previous paragraph, there exists a continuous function $f: X \to F$ (that depends on $n$) such that for all $i \in \mathcal{I}$ and all $x \in B_i'$,
\begin{align}
\label{eqn_def_of_i}
    T^n_h \big(x, G_i(x) \big)=  \big(T^n x, G_i(x)+h_n(x) \big) = \big(T^n x, G_i(T^nx) + f(x)\big).
\end{align}

Since $f$ is uniformly continuous and $F$ is discrete, there exists $\delta > 0$ such that for all $x, x' \in X$ with $d_X(x,x') < \delta$, $f(x) = f(x')$.

To show that $(Z, T_h)$ is equicontinuous, it suffices to show that the system $(Z, T_h^n)$ is equicontinuous.  Let $\eps > 0$.  We will show that there exists $\zeta > 0$ such that for all $(x,t), (x',t') \in Z$ with $d_{X \times \T}((x,t),(x',t')) < \zeta$, for all $m \in \N$,
\begin{align}
\label{eqn_equicont_to_show}
    d_{X \times \T}(T^{nm}_h(x,t),T^{nm}_h(x',t')) < \eps.
\end{align}
This will show, by definition, that the system $(Z, T_h^n)$ is equicontinuous.

Define $\sigma=\min \big\{\eta/4,\delta/2,\eps/2, 1/(4\ell) \big\}$.

\begin{claim}[Equicontinuity constant]
\label{claim_better_origin_of_equicontinuity_const}
There exists $0 < \zeta < \sigma / 2$ such that for all $x, x' \in X$ with $d_X(x,x') < \zeta$ and for all $m \geq 0$, there exists $i \in \mathcal{I}$ such that $T^{nm}x, T^{nm} x' \in B_i'$ and
\[d_X \big( T^{n(m+1)}x, T^{n(m+1)} x' \big) + \big\| G_i(T^{n(m+1)}x) - G_i(T^{n(m+1)} x') \big\| < \sigma.\]
\end{claim}

\begin{proof}
\renewcommand{\qedsymbol}{\rotatebox{45}{\#}}
Since $\mathcal{I}$ is finite and each $G_i$ is uniformly continuous on the closure $\overline{B_i'}\subseteq B_i,$ there exists $\theta>0$ such that for all $i\in \mathcal{I}$ and all $x,x'\in \overline{B_i'}$ with $d_X(x,x')<\theta,$
\begin{align}\label{E:G}
\big\|G_i(x)-G_i(x') \big\|<\sigma/2.
\end{align}
By the equicontinuity of the family $\{T^k\}_{k \in \N},$ there exists $0<\zeta<\sigma/2$ such that for all $k\in \N$ and all $x,x'$ with $d_X(x,x')<\zeta,$ we have 
\begin{align}\label{E:Equi}
d_X(T^k x, T^k x')<\min\{\theta,\sigma/2\}.
\end{align}
Let $x,x'\in X$ with $d_X(x,x')<\zeta$ and $m\geq 0.$ As $\{x_i\}_{i\in \mathcal{I}}$ is $(\eta/16)$-dense, there exists $i\in \mathcal{I}$ such that $T^{nm}x\in B_{\eta/16}(x_i)\subseteq B_i'.$ We see then that
\[d_X(T^{nm}x',x_i)\leq d_X(T^{nm}x',T^{nm}x)+d_X(T^{nm}x,x_i)<\sigma/2+\eta/16< \eta/4,\]
whereby $T^{nm}x,$ $T^{nm}x'\in B_i'.$

For the second conclusion in the claim, using the fact that $\|T^n-\id\|_\infty<\eta/16,$ we have
\[d_X(T^{n(m+1)}x,x_i)\leq d_X(T^{n(m+1)}x,T^{nm}x)+d_X(T^{nm}x,x_i)<\eta/16+\eta/16<\eta/4,\] and 
\begin{align*}
d_X(T^{n(m+1)}x',x_i) & \leq d_X(T^{n(m+1)}x',T^{nm}x')+d_X(T^{nm}x',x_i) \\
& < \eta/16+\sigma/2+\eta/16\leq \eta/16+\eta/8+\eta/16=\eta/4,
\end{align*}
so $T^{n(m+1)}x,$ $T^{n(m+1)}x'\in B_i'.$ 

Since $d_X(x,x')<\zeta,$ the inequality in \eqref{E:Equi} implies that
\[d_X \big( T^{n(m+1)}x, T^{n(m+1)} x' \big) < \min\{\theta,\sigma/2\}.\] Since $T^{n(m+1)}x,$ $T^{n(m+1)}x'\in B_i',$ using \eqref{E:G}, we have that
\[d_X \big( T^{n(m+1)}x, T^{n(m+1)} x' \big) + \big\| G_i(T^{n(m+1)}x) - G_i(T^{n(m+1)} x') \big\| < \sigma/2+\sigma/2=\sigma,\]
as was to be shown.
\end{proof}

Let $\zeta$ be the equicontinuity constant from \cref{claim_better_origin_of_equicontinuity_const}.  Let $(x,t), (x',t') \in Z$ be such that $d_{X \times \T}\big((x,t),(x',t') \big) < \zeta$.  We will show that for all $m \geq 0$,
\begin{align}
\label{eqn_induct_hypoth}
    d_{X \times \T}\big(T^{nm}_h(x,t),T^{nm}_h(x',t')\big) < \sigma
\end{align}
implies that
\begin{align}
\label{eqn_induct_conclusion}
d_{X \times \T}\big(T^{n(m+1)}_h(x,t),T^{n(m+1)}_h(x',t')\big) < \sigma.
\end{align}
Note that \eqref{eqn_induct_hypoth} holds when $m=0$ since $\zeta < \sigma$. Thus, the inequality in \eqref{eqn_equicont_to_show} follows from a simple induction on $m$ and the fact that $\sigma < \eps$.

Suppose that $m \geq 0$ and that \eqref{eqn_induct_hypoth} holds. Since $d_{X \times \T}\big((x,t),(x',t')\big) < \zeta$, it follows by \cref{claim_better_origin_of_equicontinuity_const} that there exists $i \in \mathcal{I}$ such that $T^{nm}x$, $T^{nm} x' \in B_i'$.
Since $T^{nm}_h(x,t), T^{nm}_h(x',t') \in Z$, there exist $c, c' \in F$ such that $t + h_{nm}(x) = G_i(T^{nm}x) + c$ and $t' + h_{nm}(x') = G_i(T^{nm}x') + c'$.  It follows from \cref{claim_local_lifts} (note that $T^{nm}x, T^{nm} x' \in B_i'$) and \eqref{eqn_induct_hypoth} (note that $\sigma < 1/(4\ell)$) that
\begin{align*}
    \big\|c - c' \big\| \leq \big\|G_i(&T^{nm}x) - G_i(T^{nm}x') \big\| \ + \\
    &\big\| \big(t + h_{nm}(x) \big) - \big(t' + h_{nm}(x') \big) \big\| < \frac{1}{2\ell},
\end{align*}
whereby $c = c'$.

Now we compute, using \eqref{eqn_def_of_i},
\begin{align*}
    T^{n(m+1)}_h(x,t) &= T^{n}_h T^{nm}_h(x,t) \\
    &= T^{n}_h \big(T^{nm}x, t + h_{nm}(x) \big) \\
    &= T^{n}_h \big(T^{nm}x, G_i(T^{nm}x) + c \big) \\
    &= \big(T^{n(m+1)} x, c + G_i(T^{n(m+1)}x) + f(T^{nm}x) \big).
\end{align*}
The same equalities hold with $x$ and $t$ replaced by $x'$ and $t'$, respectively.  Since $d_X \big(T^{nm}x, T^{nm}x' \big) < \sigma < \delta$, $f(T^{nm}x) = f(T^{nm}x')$. Therefore, by \cref{claim_better_origin_of_equicontinuity_const},
\begin{align*}
    d_{X \times \T} \big( &T^{n(m+1)}_h(x,t), T^{n(m+1)}_h(x',t') \big) = \\
    &d_X \big(T^{n(m+1)}x, T^{n(m+1)}x'\big) + \| G_i(T^{n(m+1)}x) - G_i(T^{n(m+1)}x')\| < \sigma,
\end{align*}
verifying \eqref{eqn_induct_conclusion} and finishing the proof of the theorem.
\end{proof}

The \emph{Kronecker factor} of a system is its largest equicontinuous factor.  In the next theorem, we prove a general result concerning minimality and the Kronecker factor of skew products on the 2-torus.  We will need this result to verify property \eqref{item_minimal_and_correct_kronecker} in \cref{maintheorem_hidden_frequencies}.  The supremum norm on $C(\T,\R)$ is denoted by $\| \cdot \|_\infty$.

\begin{theorem}
\label{theorem_condition_for_beta_skew_to_be_minimal}
Let $(\T, T)$ be an irrational rotation. Let $H: \T \to \R$ be continuous, and let $\beta = \int_\T H(x) \ dx$. If the sequence $m \mapsto \|H_m - m\beta \|_{\infty}$, where $H_m = \sum_{i=0}^{m-1} H \circ T^i$, is unbounded, then the skew product system $(\T^2,T_H)$ is minimal and has Kronecker factor $(\T, T)$.
\end{theorem}

\begin{proof}
Let $\alpha \in \T \setminus \Q$ such that $T(x) = x+ \alpha$.  First, we will show that the system $(\T^2,T_H)$ is minimal.  Suppose for a contradiction that it is not.  By \cref{theorem_not_minimal_implies_equicontinuous}, the system $(\T^2,T_H)$ is equicontinuous, hence \cref{lemma_equicontinuous_return_times} gives that the set
\[A \defeq \big\{ m \in \N \ \big| \ \sup_{(x,t) \in \T^2} d_{\T^2} \big( (x,t), T_H^m (x,t) \big) < 1/2 \big\},\]
is a \bz{} set, and hence is syndetic.  Fix $m \in A$.  For all $x \in X$, the fact that $m \in A$ implies that $\|H_m(x)\| < 1/2$. Thus, there exists a function $z_m: \T \to \Z$ such that for all $x \in \T$, $| H_m(x) - z_m(x) | < 1/2$.  Since $H_m: \T \to \R$ is continuous, the function $z_m$ must be constant: for all $x \in \T$, $z_m(x) = z_m$.  Since $H_m$ has mean $m \beta$, we get that $|m \beta - z_m| < 1/2$, which implies that $\| H_m - m \beta \|_\infty < 1$.  Using the fact that $H_{m_1 + m_2} = H_{m_1} + H_{m_2} \circ T^{m_1}$, it is easy to show that since the sequence $m \mapsto \|H_m - m\beta \|_{\infty}$ is bounded along a syndetic subsequence, it is bounded.  This is in contradiction to our assumption, concluding the proof that the system $(\T^2, T_H)$ is minimal.

Now we will show that the system $(\T^2, T_H)$ has Kronecker factor $(\T, T: x \mapsto x+ \alpha)$.  Since $(\T^2, T_H)$ is minimal and distal, its Kronecker factor is determined by the regional proximal relation \cite[Theorem 1.1]{veech1968}: $(x,t)$ is regionally proximal to $(y,s)$ if and only if for all $\eps > 0$, there exists $(z, u)$ with $\|(z,u) - (x,t) \| < \eps$ and $m \in \N$ such that $\| T_H^m (z,u) - (x,t) \| < \eps$ and $\| T_H^m (x,t) - (y,s) \| < \eps$.  Because the factor $(\T,T)$ is equicontinuous, if $(x,t)$ and $(y,s)$ are regionally proximal, then $x=y$.

Because $T_H$ commutes with rotation in the second coordinate of $\T^2$, to prove that the system $(\T^2, T_H)$ has Kronecker factor $(\T, T)$, it suffices to prove that for all $x, t \in \T$, the points $(x,0)$ and $(x,t)$ are regionally proximal.  Let $x, t \in \T$, and let $\eps > 0$.  Let $0 < \delta < \eps$ be sufficiently small so that if $\|x - y\| < \delta$, then $|H(x) - H(y)| < \eps$.  Because $(\T^2, T_H)$ is minimal, the set of return times of $(x,0)$ to the $\delta$-neighborhood of the point $(x,t)$ is syndetic; therefore, the set
\[B \defeq \big\{m \in \N \ \big| \ \|m \alpha\| < \delta, \ \|H_m (x) - t \| < \eps \big\}\]
is syndetic.  Since $B$ is syndetic, it follows from our assumptions that the sequence $m \mapsto \|H_m - m\beta \|_{\infty}$ is unbounded along $B$.

Let  $N \in \N$ be such that $\{ n \alpha \ | \ 1 \leq n \leq N \}$ is $\eps$-dense in $\T$, and choose $m \in B$ such that $\|H_m - m\beta \|_{\infty} > 2N$.  We will show that there exists $(z,0) \in \T^2$ such that $\| (z,0) - (x,0) \| < \eps$ and $\| T_H^m (z,0) - (x,0)\| < \eps$.  Since $m \in B$, we will have that $\| T_H^m (x,0) - (x,t)\| < \eps$, and since $\eps > 0$ was arbitrary, this will finish the proof that $(x,0)$ and $(x,t)$ are regionally proximal.

Since $\|H_m - m\beta \|_{\infty} > 2N$ and the mean of $H_m$ is $m \beta$, there exist $x_\star, x^{\star} \in \T$ such that $H_m(x_\star) - m \beta < -2N$ and $H_m(x^\star) - m \beta > 2N$.  By our choice of $\delta$, for all $x \in \T$,
\begin{align}
\label{eqn_jump_in_hm_is_not_large}
    |H_m (x + \alpha) - H_m(x)| = |H(x + m \alpha) - H(x)| < \eps.
\end{align}
By our choice of $N$, there exist $n_1, n_2 \in \{1, \ldots, N\}$ such that $x_\star + n_1 \alpha, x^\star + n_2 \alpha$ are both within $\eps$ of $x$. By repeatedly appealing to \eqref{eqn_jump_in_hm_is_not_large}, we have that $|H_m(x_\star + n_1 \alpha) - H_m(x_\star) | < N \eps$, from which it follows that $H_m(x_\star + n_1 \alpha) - m \beta < -N$.  Similarly, $H_m(x^\star + n_2 \alpha) - m \beta > N$.  Since $H_m$ is continuous, by the intermediate value theorem, the image of $H_m$ restricted to an $\eps$-ball about $x$ is all of $\T$.  Therefore, there exists $z \in \T$, $\|z - x\| < \eps$, such that $H_m(z) = 0$.  It follows that the point $(z,0) \in \T^2$ satisfies $\| (z,0) - (x,0) \| < \eps$ and $\| T_H^m (z,0) - (x,0)\| < \eps$, as was to be shown.
\end{proof}

\begin{remark}
\label{rmk_gottschalk_hedlund_connection}
\cref{theorem_condition_for_beta_skew_to_be_minimal} is a complement to the classic theorem of Gottschalk and Hedlund \cite[Theorem 4.11]{Gottschalk_Hedlund_book}, which asserts that the sequence $m \mapsto \|H_m - m\beta\|_{\infty}$ is bounded if and only if there exists a continuous function $G: \T \to \R$ such that $H(x) = G(x+\alpha) - G(x) + \beta$.  In this case, the $\T^2$-homeomorphism $(x,y) \mapsto (x, y + G(x))$ demonstrates the topological conjugacy between the skew product system $(\T^2, T_H)$ and the rotation $\big(\T^2, (x,y) \mapsto (x+ \alpha, y+\beta)\big)$.  Therefore, if the sequence $m \mapsto \|H_m - m\beta\|_{\infty}$ is bounded, the skew product system $(\T^2,T_H)$ is equicontinuous, and it is minimal if and only if $1$, $\alpha$, and $\beta$ are linearly independent over the rationals.
\end{remark}

\subsection{The hidden frequencies example: a proof of \cref{maintheorem_hidden_frequencies}}
\label{sec_hidden_frequencies}

In this section, we prove \cref{maintheorem_hidden_frequencies} by giving an example of a minimal skew product system on $\T^2$ in which recurrence in the Kronecker factor (the system's largest equicontinuous factor, the base rotation) does not suffice for recurrence in the system.  Such an example stands in sharp contrast to other systems in which the answer to Katznelson's Question is known, and it demonstrates at least some of the difficulty of answering the question for more general systems. We use the notation of skew product systems from \cref{def_skew_prod_systems}.

Define $\tilde{H}: [0,1] \to \R$ by
\[\tilde{H}(x) = x^4 - 2x^3 + x^2 - 1/30.\]
Since $\tilde{H}(0) = \tilde{H}(1)$, there exists $H \in C(\T, \R)$ so that $\tilde H = H \circ \restr{\pi}{[0,1]}$, where $\pi: \R \to \T$ is the quotient map. The system in the proof of \cref{maintheorem_hidden_frequencies} will be a skew product system of the form $(\T^2,T_{H + \beta})$,
\[T_{H+\beta} (x,y) = (x + \alpha, y + H(x) + \beta),\]
for certain $\alpha, \beta \in \T$.  Recall from \cref{def_skew_prod_systems} that in the system $(\T^2,T_{H + \beta})$, the map $H$ is implicitly pre-composed with the quotient map $\pi$.

\begin{lemma}\leavevmode
\label{lemma_basic_facts_about_H}
For all $x \in \T$, for all $n \geq 4$,
\begin{align}
\label{eqn_riemann_sum_estimate}
    \left| \sum_{i=0}^{n-1} H \left(x + \frac in \right)\right| < \frac 1{12}.
\end{align}
\end{lemma}

\begin{proof}
Let $n \geq 4$. Since the map $x \mapsto \sum_{i=0}^{n-1} H(x + i /n )$ is $(1/n)$-periodic on $\T$, it suffices to verify \eqref{eqn_riemann_sum_estimate} with $H$ replaced by $\tilde{H}$ for all real values of $x \in [0,1/n]$. By the Faulhaber formulae for $\sum_{i=0}^{n-1}i^k$ for $k=1,2,3,4$, we have that 
\begin{eqnarray*}
\tilde{H} \left(x + \frac in \right) & = & x^4+\left(-2+\frac{4}{n}i\right)x^3+\left(1-\frac{6}{n}i+\frac{6}{n^2}i^2\right)x^2 \\
& + & \left(\frac{2}{n}i-\frac{6}{n^2}i^2+\frac{4}{n^3}i^3\right)x+\left(-\frac{1}{30}+\frac{1}{n^2}i^2-\frac{2}{n^3}i^3+\frac{1}{n^4}i^4\right).
\end{eqnarray*}
Skipping the algebra, the sum of interest is equal to
\[\sum_{i=0}^{n-1} \tilde{H} \left(x + \frac in \right) =  nx^4-2x^3+\frac{1}{n}x^2-\frac{1}{30n^3}.\]
The claim follows since, for every $x\in [0,1/n],$ 
\[\left| \sum_{i=0}^{n-1} \tilde{H} \left(x + \frac in \right)\right|\leq nx^4+2x^3+\frac{1}{n}x^2+\frac{1}{30n^3}\leq \frac{121}{30n^3}<\frac{1}{12}.\]
\end{proof}

The following theorem gives a sufficient condition for a skew-product system of the form $(x,y) \mapsto \big(x + \alpha, y + H(x) \big)$ on $\T^2$ to be minimal. It is quick to check that the function $\tilde H$ defined above satisifies the hypotheses of the theorem; in fact, this is precisely how $\tilde H$ was chosen.

\begin{theorem}[{\cite[Theorem 1.4]{Hellekalek_Larcher_weyl_sums_and_skew_products_1989}, combined with the remark following it}]
\label{theorem_Hellekalek_Larcher_sufficient_condition_for_minimality}
Suppose that $\tilde{H}: [0,1] \to \R$, $K \geq 2$, and $\alpha \in [0,1]$ satisfy
\begin{itemize}
    \item $\tilde{H}$ is $K$-times continuously differentiable on $[0,1]$;
    \item $\int_0^1 \tilde{H}(x) \ dx = 0$;
    \item for all $0 \leq j \leq K-2$, $\tilde{H}^{(j)}(0) = \tilde{H}^{(j)}(1)$;
    \item $\tilde{H}^{(K-1)}(0) \neq \tilde{H}^{(K-1)}(1)$; and 
    \item for all $i \in \N$, $a_{i+1} \geq q_i^{2K}$, where $a_i$ is the $i^{\text{th}}$ partial quotient of $\alpha$ (as in \eqref{eqn_alpha_as_continued_fraction} below) and $q_i$ is the denominator of the $i^{\text{th}}$ convergent of $\alpha$.
\end{itemize}
Defining $H \in C(\T, \R)$ such that $\tilde H = H \circ \restr{\pi}{[0,1]}$, the skew-product system $(x,y) \mapsto \big(x + \alpha, y + H(x) \big)$ on $\T^2$ is minimal.
\end{theorem}

The following lemma explains the choice of $\alpha$ in the system we construct; its second conclusion follows from \cref{rmk_gottschalk_hedlund_connection} and \cref{theorem_Hellekalek_Larcher_sufficient_condition_for_minimality}.

\begin{lemma}
\label{lemma_selection_of_alpha_from_schoissengeier}
There exists $\alpha \in \R \setminus \Q$ and a sequence $(n_i)_{i \in \N} \subseteq \N$ for which:
\begin{enumerate}
    \item \label{item_alpha_liminf_statement} the nearest integer to $n_i \alpha$ is coprime to $n_i$ and $\lim_{i \to \infty} n_i \|n_i \alpha\| = 0$; and
    \item \label{item_alpha_unboundedness} the sequence $m \mapsto \| H_m \|_\infty$ is unbounded.
\end{enumerate}
\end{lemma}

\begin{proof}
Let $(a_i)_{i \in \N} \subseteq \N$ be a sufficiently rapidly increasing sequence such that, on defining $n_0 = 1$, $n_1 = a_1$, and $n_i = a_i n_{i-1} + n_{i-2}$, we have for all $i \in \N$ that $a_{i+1} \geq n_i^8$. (Take, for example, $a_i = 10^{10^i}$.) Let $\alpha \in \R$ be the real number whose sequence of simple continued fraction partial quotients is $(a_n)_{n \in \N}$:
\begin{align}
\label{eqn_alpha_as_continued_fraction}
    \alpha = \frac{1}{a_1 + \frac{1}{a_2 + \frac{1}{\ddots}}}.
\end{align}
We claim that this $\alpha$ and the sequence $(n_i)_{i \in \N}$ satisfy the conclusion of the lemma.

Denote by $p_i / q_i$ the $i^{\text{th}}$ continued fraction convergent of $\alpha$.  The following are standard facts in the theory of continued fractions \cite[Chapter 1]{khintchine_continued_fractions_book_1963}: $q_i = n_i$ (where $n_i$ is as defined in the previous paragraph); $n_i$ and $p_i$ are coprime; and 
\[n_i | n_i \alpha - p_i | = n_i \|n_i \alpha\| < \frac{n_i}{n_{i+1}} < \frac{1}{a_{i+1}},\]
in particular, the nearest integer to $n_i \alpha$ is $p_i$. This shows that the condition in \eqref{item_alpha_liminf_statement} is satisfied.

According to \cref{theorem_Hellekalek_Larcher_sufficient_condition_for_minimality} (with $K = 4$), the skew product system $(\T^2, T_H)$ is minimal.  It follows from \cref{rmk_gottschalk_hedlund_connection} (where $\beta = 0$) and the fact that the rotation $(x,y) \mapsto (x + \alpha, y)$ is not minimal that the sequence $m \mapsto \| H_m \|_\infty$ must be unbounded, as was to be shown.
\end{proof}

\begin{lemma}
\label{lemma_selection_of_beta_from_fatou}
For any sequence $(n_i)_{i \in \N} \subseteq \N$, there exists $\beta \in \R \setminus \Q$ such that for infinitely many $i \in \N$, $\| n_i \beta \| > 1/3$.
\end{lemma}

\begin{proof}
    Define
    $
    D_i\coloneqq\{\beta\in\T\ | \ \|n_i\beta\|\in (1/3,1/2]\}.
    $
    Let $\mu$ be the Lebesgue measure on $\T$. Since $\mu(D_i)=1/3$ for all $i\in\N$, it follows from Fatou's lemma that the set
    \[
    \{\beta\in\T\ | \ \beta\in D_i\text{ for infinitely many }i\in\N\}
    \]
    has measure at least $1/3$. Any irrational $\beta$ in this set satisfies the conclusion of the lemma.
\end{proof}

\begin{proof}[Proof of \cref{maintheorem_hidden_frequencies}]

Let $\alpha \in \R \setminus \Q$ and $(n_i)_i \subseteq \N$ be as guaranteed by \cref{lemma_selection_of_alpha_from_schoissengeier}.   Appealing to \cref{lemma_selection_of_beta_from_fatou}, let $\beta \in \R \setminus \Q$ and pass to a subsequence of $(n_i)_{i \in \N}$ so that for all $i \in \N$, $\|n_i \beta \| > 1/3$.  Let $L > 0$ be a Lipschitz constant for $H$, and let $0 < \delta < 1/(12L)$.  Passing to a further subsequence of $(n_i)_{i\in \N}$, we may assume that for all $i \in \N$, $n_i \|n_i \alpha\| < \delta$ and $\| n_i \beta \| > 1/3$.

It follows from \cref{lemma_selection_of_alpha_from_schoissengeier} and \cref{theorem_condition_for_beta_skew_to_be_minimal} that the system $(\T^2,T_{H+\beta})$ is minimal and has Kronecker factor $(\T, x \mapsto x + \alpha)$.  Put $R = \{ n_i \ | \ i \in \N \}$.  Since $\lim_{i \to \infty} \|n_i \alpha\| = 0$, the set $R$ is a set of recurrence for $(\T, x \mapsto x + \alpha)$. We have only left to verify that $R$ is not a set of recurrence for $(\T^2,T_{H + \beta})$.  It suffices to show that for all $x \in \T$ and all $m \in R$, $\|H_m(x) + m \beta\| > 1/6$.

Let $x \in \T$ and $m \in R$, and let $k$ be the nearest integer to $m \alpha$ so that $|\alpha - k/m| < \delta / m^2$.  Since $k$ and $m$ are coprime, there exists a permutation $\sigma$ of $\{0, \ldots, m-1\}$ such that for all $i \in \{0, \ldots, m-1\}$, $\| i \alpha - \sigma(i) / m\| < \delta / m$.  We estimate
\begin{align*}
    \left| H_m(x) - \sum_{i=0}^{m-1} H\left( x + \frac im \right)\right| &= \left| \sum_{i=0}^{m-1} \left( H(x + i \alpha) - H \left(x + \frac{\sigma(i)}{m} \right) \right)\right| \\
    &\leq  \sum_{i=0}^{m-1} \left| H(x + i \alpha) - H \left(x + \frac{\sigma(i)}{m} \right) \right| \\
    &\leq  \sum_{i=0}^{m-1} L \left\|  i \alpha  -  \frac{\sigma(i)}{m} \right\| \leq L \delta.
\end{align*}
It follows from \eqref{eqn_riemann_sum_estimate} that
\[\big|H_m(x) \big| \leq L \delta + \frac 1{12} < \frac 16.\]
Since $\| m \beta \| > 1/3$, we have that $\|H_m(x) + m \beta \| > 1/6$, as was to be shown.  This concludes the proof of \cref{maintheorem_hidden_frequencies}.
\end{proof}

\section{Katznelson's Question in a combinatorial framework}
\label{sec_combinatorial_setting}

As recounted in \cref{sec_history}, combinatorial forms of Katznelson's Question and its relatives were considered long before Katznelson and others popularized them in dynamical form.  In this section, we provide some of those combinatorial formulations and prove the equivalence between them.  We also prove \cref{maintheorem_combinatorial_form}, a combinatorial corollary to our main dynamical result, \cref{maintheorem_katznelson_for_skew_towers}.

Recall that a set $A\subseteq \N$ is syndetic if there exists a finite set $F \subseteq \N$ such that $A-F \supseteq \N$.  The set $A$ is \emph{piecewise syndetic} if there exists a finite set $F \subseteq \N$ such that $A-F$ contains arbitrarily long intervals (i.e., is \emph{thick}).

\subsection{Combinatorial forms of Katznelson's Question}
\label{sec_comb_forms_of_katznelson}

In what follows, the phrase ``Question A implies Question B'' means that a positive answer to Question A yields a positive answer to Question B.  We say that Questions A and B are \emph{equivalent} if $A$ implies $B$ and $B$ implies $A$. We will prove the equivalence between the various forms of Katznelson's Question posed in \cref{sec_overview} indirectly, beginning first with some alternate combinatorial formulations.

\begin{remark}
For most of the questions posed in this paper, there is not a material difference between the set of positive integers $\N$ and the set of all integers $\Z$.  Some combination of the following three facts generally suffices to prove the equivalence between analogous questions in these two settings: (1) a Bohr neighborhood of zero is symmetric about 0 and, when restricted to $\N$, is a \bz{} set; (2) if $A \subseteq \N$ is a \bz{} set, then the set $A \cup (-A) \cup \{0\}$ contains a Bohr neighborhood of zero; and (3) if $A$ is a syndetic subset of $\N$, then $A \cup (-A)$ is syndetic in $\Z$, and if $A$ is syndetic in $\Z$, then $A \cap \N$ is syndetic in $\N$.
\end{remark}

Since syndetic sets are piecewise syndetic, a positive answer to the following question implies a positive answer to the combinatorial form of Katznelson's Question.

\begin{namedthm}{Question C1}
    If $A \subseteq \N$ is piecewise syndetic, does its set of differences $A-A$ contain a Bohr neighborhood of zero?
\end{namedthm}

To see that Question C1 implies the combinatorial form of Katznelson's Question, we will use the fact that if $A \subseteq \N$ is piecewise syndetic, then it is \emph{broken syndetic}: there exists a syndetic set $S \subseteq \N$ with the property that for all finite $F \subseteq S$, there exists $n \in \N$ such that $n+F \subseteq A$; see \cite[Proof of Theorem 1]{ruzsa_difference_sets_unpublished_1985}.  It follows immediately that $S-S \subseteq A-A$, and hence that $S-S$ being a \bz{} implies that $A-A$ is a \bz{} set.

While the property of being syndetic is not partition regular, the related notion of piecewise syndeticity is; see \cite[Theorem 1.24]{furstenberg_book_1981}. Since one cell of any finite partition of $\N$ is piecewise syndetic, Question C1 implies the following useful combinatorial form of Katznelson's Question.

\begin{namedthm}{Question C2}
    If $\N = \cup_{i=1}^r A_i$ is a finite partition of $\N$, does the set $\cup_{i=1}^r (A_i - A_i)$ contain a Bohr neighborhood of zero?
\end{namedthm}

To see that Question C2 implies C1, it suffices to show that Question C2 implies Katznelson's Question. If $A \subseteq \N$ is syndetic, then there exists $r \in \N$ such that $\N = \cup_{i=1}^r (A - i)$.  It follows from C2 that the set
\[\bigcup_{i=1}^r \big( (A - i) - (A - i)\big) = A-A\]
is a \bz{} set, yielding a positive answer to Katznelson's Question.  Thus, Questions C1 and C2 are equivalent forms of Katznelson's Question.  This equivalence is also proved using different terminology in \cite[Theorem~1]{ruzsa_difference_sets_unpublished_1985}.

\begin{remark}
The foregoing sequence of questions may lead one to wonder whether or not the family of subsets of $\N$ that have \bz{} large differences is partition regular.  This is not the case, as can be seen by the following example.  Let $A \subseteq \N$ be a set of positive upper asymptotic density that does not have \bz{} large differences; such a set exists by an example of K\v{r}\'{i}\v{z} \cite{kriz1987}.  The set of differences of $A$ is syndetic \cite[Proposition 3.19]{furstenberg_book_1981}, so there exists $\ell \in \N$ such that $A-A-\{1, \ldots, \ell\} = \N$.  Put $B = A \cup (A+1) \cup \cdots \cup (A+\ell)$.  Because $B-B = \N$, the set $B$ has \bz{} large differences, but no cell $A + i$ of the partition of $B$ has \bz{} large differences.
\end{remark}

\begin{lemma}
Katznelson's Question is equivalent to its combinatorial formulation.
\end{lemma}

\begin{proof}
We will show the equivalence between Question D1 (from \cref{sec_dynamically_equivalent_forms}) and Question C2.  This equivalence has been documented a number of times in the literature; see, for example, \cite[Lemma 4.5]{boshernitzanglasner2009} or \cite[Proof of Theorem 2.1]{KatznelsonChromaticNumber2001}.  Since the argument is short, we provide it here for completeness.

To see that Question C2 implies Question D1, suppose $(X,T)$ is a system and $\eps > 0$.  Let $X = \cup_{i=1}^r B_i$ be a cover of $X$ by finitely many balls of diameter less than $\eps$.  Fix $x \in X$, and pull the cover of $X$ back through the map $n \mapsto T^n x$ to a cover $\N = \cup_{i=1}^r A_i$ so that $n \in A_i$ implies that $T^n x \in B_i$.  It is quick to check that $\cup_{i=1}^r (A_i - A_i) \subseteq \dyret_\eps(X,T)$, whereby a positive answer to C2 implies a positive answer to D1.

That Question D1 implies C2 relies on a correspondence principle.  Suppose $\N = \cup_{i=1}^r A_i$, and let $f: \N \to \{1, \ldots, r\}$ be such that for all $n \in \N$, $n \in A_{f(n)}$.  The sequence $f$ belongs to the compact metric space $(\{1, \ldots, r\}^\N, d)$, on which we consider the left-shift map $T$.  Put $X = \overline{ \{ T^n f \ | \ n \in \N \} }$. Let $\eps > 0$ be such that if $x, y \in X$ satisfy $d_X(x, y) < \eps$, then $x(1) = y(1)$. If $m \in \dyret_\eps(X,T)$, then there exists $x \in X$ such that $x(1) = x(m+1)$. Since $f$ has a dense orbit in $X$, there exists $n \in \N$ such that $f(n+i) = x(i)$ for $i = 1, \ldots, m+1$.  It follows that $f(n+1) = f(n+m+1)$, whereby $n+1$ and $n+m+1$ are in the same element of the cover.  Therefore, $m \in \cup_{i=1}^r (A_i - A_i)$.  We've shown that $\dyret_\eps(X,T) \subseteq \cup_{i=1}^r (A_i - A_i)$, whereby a positive answer to D1 implies a positive answer to C2.
\end{proof}

In \cite{KatznelsonChromaticNumber2001}, Katznelson's Question appears in terms of chromatic numbers of certain graphs on $\N$. For $R \subseteq \N$, define a graph $G_R$ on $\N$ by putting an edge between $n, m \in \N$ if and only if $|n-m| \in R$.  Denote by $\chi(G_R)$ the chromatic number of the graph $G_R$.

\begin{namedthm}{Question C3}[{\cite{KatznelsonChromaticNumber2001}}]
    If $R$ is a \bzs{} set, is $\chi(G_R) = \infty$?
\end{namedthm}

This question is quickly seen to be a reformulation of Question C2, and hence of Katznelson's Question.  Indeed, a finite partition $\N = \cup_{i=1}^r A_i$ is exactly a finite coloring of $\N$.  By \cref{rmk_equiv_forms_of_bohr_and_bohrzero}, the set $\cup_{i=1}^r (A_i - A_i)$ is a \bz{} set if and only if for all \bzs{} sets $R \subseteq \N$,
\begin{align}
\label{eqn_bzs_intersects_union_of_differences}
    R \cap \bigcup_{i=1}^r (A_i - A_i) \neq \emptyset.
\end{align}
Note that \eqref{eqn_bzs_intersects_union_of_differences} holds if and only if there are adjacent vertices in the graph $G_R$ with the same color.  Thus, both C2 and C3 ask whether or not \eqref{eqn_bzs_intersects_union_of_differences} holds for all finite colorings of $\N$ and all \bzs{} sets $R \subseteq \N$.

The following question appears easier to answer than the combinatorial form of Katznelson's Question; it is, in fact, shown to be equivalent by an elementary argument.  Note that it was shown in \cite{elliskeynes1972} that the set $A+A+A$ contains a Bohr set when $A \subseteq \Z$ is syndetic.

\begin{namedthm}{Question C4}
    If $A \subseteq \Z$ is syndetic, contains 0, and satisfies $A = -A$, is the triple sumset $A + A + A$ a \bz{} set?
\end{namedthm}

To see that Katznelson's Question implies Question C4, suppose $A \subseteq \Z$ has the properties stipulated in C4.  Since $A-A \subseteq A+A+A$, if the combinatorial form of Katznelson's Question has a positive answer, then $A+A+A$ is a \bz{} set.  To see that Question C4 implies Katznelson's Question, we borrow a clever argument from \cite[Theorem 2]{ruzsa_difference_sets_unpublished_1985}; see also \cite[Proof of Theorem 2.1]{hegyvariruzsa2016}.  Suppose the combinatorial form of Katznelson's Question has a negative answer: there exists a syndetic set $A \subseteq \Z$ for which $A-A$ is not a \bz{} set. Put
\[A_0 = (4A + 1) \cup -(4A + 1) \cup \{0\}.\]
Clearly $A_0$ is syndetic, contains 0, and satisfies $A_0 = -A_0$. Considering residues modulo 4, it is quick to show that
\[(A_0 + A_0 + A_0)/4 \subseteq A-A.\]
Since $A-A$ is not a \bz{} set, neither is the set $(A_0 + A_0 + A_0) / 4$.  By \cref{lemma_dilates_of_bohr_are_bohr}, it follows that the set $A_0 + A_0 + A_0$ is not a \bz{} set, answering Question C4 in the negative.

There is a useful dialogue between sequences (more generally, functions) and dynamics; see \cite[Theorem 5.6]{boshernitzanglasner2009} for a particular connection between sequences and recurrence and \cite{weiss_single_orbit_book} for a broader view.  This is part of the reason why the sequential formulation of Katznelson's Question echoes the dynamical one.

\begin{lemma}
\label{lemma_katz_equivalent_to_sequential_katz}
Katznelson's Question is equivalent to its sequential formulation.
\end{lemma}

\begin{proof}
It is easier to see the equivalence between the sequential form of Katznelson's Question and Question C2.  To see that the former implies the latter, suppose $\N = \cup_{i=1}^r A_i$; without loss of generality, we may assume that the sets $A_i$ are disjoint.  Choose $2r+1$ distinct points on the 1-torus, $t_{-r}, \dots, t_r \in \T$, and define $f: \Z \to \T$ by $f(0)= t_0$ and, for $n \in A_i$, $f(n) = t_i$ and $f(-n) = t_{-i}$.  If $\eps < \min_{i \neq j} \| t_i - t_j\|$, then
\[\big\{ m \in \Z \ | \ \inf_{n \in \Z} \| f(n+m) - f(n) \| < \eps \big\} \subseteq \bigcup_{i=1}^r (A_i - A_i).\]
Thus, a positive answer to the sequential form of Katznelson's Question implies a positive answer to Question C2.

To show the converse, let $f: \Z \to \T$ and $\eps > 0$. Cover $\T$ by finitely many balls of diameter $\eps$: $\T = \cup_{i=1}^r B_i$.  Pull this cover of $\T$ back through $f$ to get a finite partition $\Z = \cup_{i=1}^r A_i$, and restrict this partition to one of $\N$.  It is quick to check that
\[\bigcup_{i=1}^r (A_i - A_i) \subseteq \big\{ m \in \Z \ | \ \inf_{n \in \Z} \| f(n+m) - f(n) \| < \eps \big\},\] whereby a positive answer to C2 yields a positive answer to the sequential form of Katznelson's Question.
\end{proof}

\begin{remark}
It is clear from the argument that the converse implication {in \cref{lemma_katz_equivalent_to_sequential_katz}} depends only on the total boundedness of $\T$.  Thus, {one} can formulate an equivalent question that appears more difficult to answer by replacing $\T$ by an arbitrary totally bounded metric space in the sequential formulation of Katznelson's Question.
\end{remark}

\begin{remark}
\label{rmk_special_case_of_seq_form}
The following special case of the sequential form of Katznelson's Question is, in fact, equivalent to the more general form stated in the introduction: \emph{Is it true that for all $\alpha \in \R$ and all $\eps>0$, the sequence $f: n \mapsto \{2^n \alpha\}$ is such that the set
\begin{align*}
    \big\{ m \in \Z \ | \ \inf_{n \in \Z} \| f(n+m) - f(n) \| < \eps \big\}
\end{align*}
contains a Bohr neighborhood of zero?} To see that the two are equivalent, it suffices by \cref{lemma_katz_equivalent_to_sequential_katz} to show that a positive answer to the special case yields a positive answer to Question C2.

Suppose $\N = \cup_{i=1}^r A_i$, and define $c: \N \to \{1, \ldots, r\}$ so that for all $n \in \N$, $n \in A_{c(n)}$.  Choose $k \in \N$ so that $2^k > 2r$, and define
\[\alpha = \sum_{n = 1}^\infty \frac{2 c(n)}{2^{kn}}.\]
Note that if $m, n \in k\N$ and $\| 2^{n+m}\alpha - 2^n \alpha \| < 1/2^{k}$, then $c\big( (n+m) / k \big) = c(n/k)$. It follows by the special case of the sequential form of Katznelson's Question stated above and by \cref{lemma_dilates_of_bohr_are_bohr} that the set
\[B \defeq \big\{ m \in \N \ | \ \inf_{n \in \N} \| 2^{n+m}\alpha - 2^n \alpha \| < 1/2^{2k} \big\} \cap k \N \]
is a \bz{} set.  Note that if $n \in \N$ is such that $\| 2^{n+m}\alpha - 2^n \alpha \| < 1/2^{2k}$, then there exists $n' \in k \N$ such that $\| 2^{n'+m}\alpha - 2^{n'} \alpha \| < 1/2^{k}$.  Thus, for all $m \in B$, there exists $n \in k\N$ such that $c\big( (n+m) / k \big) = c(n/k)$.  This implies that the set $B/k$, which is a \bz{} set by \cref{lemma_dilates_of_bohr_are_bohr}, is contained in the set $\cup_{i=1}^r (A_i - A_i)$, yielding a positive answer to Question C2.
\end{remark}

\begin{remark}
A minimal system $(X,T)$ has \bz{} large returns if and only if for all $\varphi: X \to \R$ continuous and all $x \in X$, the observable sequence $f: n \mapsto \varphi(T^n x)$ is such that
\[\big\{ m \in \N \ | \ \inf_{n \in \N} \| f(n+m) - f(n) \| < \eps \big\}\]
is a \bz{} set. We will not have use for this connection explicitly in this paper, so we leave the verification of this fact to the curious reader.
\end{remark}

We conclude this section with a formulation of Katznelson's Question in terms of the lengths of zero-sum blocks of cyclic-group-valued sequences. A \emph{zero-sum-block} for $f: \Z \to \Z / k \Z$ is an interval on which $f$ sums to zero; this notion does not appear to be well-studied, but did appear recently in the literature \cite{caro_hansberg_montejano_zero_sum_subsequences2016}.

\begin{namedthm}{Question C5}
    If $f: \Z \to \Z / k \Z$, is the set of lengths of zero-sum blocks
    \begin{align}
    \label{eqn_set_in_q_s4}
    \big\{ m \in \N \ \big| \ \exists\; n \in \Z, \ f(n) + \cdots + f(n+m-1) = 0 \big\}
    \end{align}
     a \bz{} set?
\end{namedthm}

To see that Question C5 implies C2, suppose $\N = \cup_{i=1}^r A_i$.  Define $g: \Z \to \{-r, \dots, r\}$ by $g(0) = 0$ and, for $n \in A_i$, $g(n) = i$ and $g(-n) = -i$.  Put $k = 2r+1$, and define $f: \Z \to \Z / k\Z$ by $f(n) = g(n+1) - g(n)$ modulo $k$.  If $f(n) + \cdots + f(n+m-1) = 0$, then $g(n) = g(n+m)$, whereby $m \in A_i - A_i$ for some $i \in \{1, \ldots, r\}$.  Therefore, the set of lengths of zero-sum blocks for $f$ is contained in the set $\cup_{i=1}^r (A_i - A_i)$, implying that a positive answer to C5 yields a positive answer to C2.

To see the converse, let $f: \Z \to \Z / k \Z$.  Define $g: \N \to \Z / k \Z$ by $g(n) = \sum_{i=1}^{n-1} f(i)$.  For $i \in \{0, \ldots, k-1\}$, let $A_i$ be the set of those $n \in \N$ for which $g(n) \equiv i \pmod k$.  The set $\cup_{i=0}^{k-1} (A_i - A_i)$ is a subset of the set in \eqref{eqn_set_in_q_s4}. Since $\cup_{i=0}^{k-1} (A_i - A_i)$ is a \bz{} set, so is the set in \eqref{eqn_set_in_q_s4}.

\begin{remark}
\label{rmk_careful_with_classes}
Finite cyclic groups are not essential to the formulation of Question C5. If $G$ is a compact abelian group with invariant metric $d_G$ and $f: \Z \to G$, an \emph{$\eps$-sum-block} is an interval on which $f$ sums to within $\eps$ of the identity 0.  Using the same reasoning as above, it is easy to see that Question C5 is equivalent to the ostensibly more difficult question obtained by replacing the set in \eqref{eqn_set_in_q_s4} with
\[\big\{ m \in \N \ \big| \ \inf_{n \in \Z} d_G\big( f(n) + \cdots + f(n+m-1) , 0 \big) < \eps \big\}.\]
    
It should also be noted that a positive answer to the sequential form of Katznelson's Question for a class of sequences does not necessarily imply a positive answer to Question C5 for the same class.  For example, the sequential form of Katznelson's Question trivially has a positive answer for almost periodic sequences. It does not follow, however, from the equivalence described above that Question C5 has a positive answer for almost periodic sequences.  In fact, it is true that Question C5 has a positive answer for AP sequences, but this is the result of \cref{maintheorem_combinatorial_form} which requires additional arguments.
\end{remark}

A number of open questions closely related to Katznelson's Question are presented in \cref{sec_katznelson_adjacent}.

\subsection{Combinatorial results: a proof of \cref{maintheorem_combinatorial_form}}
\label{sec_proof_of_comb_form}

In this section, we provide positive answers to the combinatorial form of Katznelson's Question for 2-syndetic sets and for certain classes of syndetic sets that arise naturally in topological dynamics.  The first is accomplished by a simple combinatorial argument, while the second -- made precise in the statement of \cref{maintheorem_combinatorial_form} -- is derived from \cref{maintheorem_katznelson_for_skew_towers}.

\begin{theorem}
\label{theorem_two_coloring_Katznelson}
For all $2$-colorings $\N = A_1 \cup A_2$, either $(A_1 - A_1) \cup (A_2 - A_2) \supseteq \N$ or there exists $d \in \N$ for which $d \N \subseteq (A_1 - A_1) \cap (A_2 - A_2)$.  In particular, the set $(A_1 - A_1) \cup (A_2 - A_2)$ is a \bz{} set.
\end{theorem}

\begin{proof}
If $(A_1 - A_1) \cup (A_2 - A_2) \supseteq \N$, then the conclusion of the theorem holds.  Otherwise, there exists $d \in \N \setminus \big( (A_1 - A_1) \cup (A_2 - A_2)\big)$.  Since $d \notin A_1 - A_1$, we see that $A_1 + d \subseteq A_2$.  Similarly, $A_2 + d \subseteq A_1$.  It follows that $A_1 + 2d \subseteq A_1$ and $A_2 + 2d \subseteq A_2$, and hence that $2d \N \subseteq (A_1 - A_1) \cap (A_2 - A_2)$.
\end{proof}

The next corollary follows immediately from \cref{theorem_two_coloring_Katznelson}.  It is still not known whether or not the final conclusion in \cref{theorem_two_coloring_Katznelson} holds for all 3-colorings of $\N$; see \cref{question_3_syndetic_sets} in \cref{sec_next_steps}.

\begin{corollary}
\label{cor_differences_of_a_two_syndetic_set}
Let $A \subseteq \N$.  If $A \cup (A-\ell) \supseteq \N$, then there exists $d \in \N$ such that $d \N \subseteq A-A$.  In particular, the set $A-A$ is a \bz{} set.
\end{corollary}

We turn now to the proof of \cref{maintheorem_combinatorial_form}. We will make use of systems of the form $(X \times \T^k, T_{h,\id,\ldots,\id})$, where the map $T_{h,\id,\ldots,\id}: X \times \T^k \to X \times \T^k$ is defined by \[T_{h,\id,\ldots,\id}(x,t_1,\ldots,t_k)= \big(Tx,t_1+h(x),t_2+t_1,\ldots,t_k+t_{k-1} \big).\]
Thus, the system $(X \times \T^k, T_{h,\id,\ldots,\id})$ is an iterated skew product system, where the identity skewing map is repeated $k-1$ many times, over a base skew product system $(X \times \T, T_h)$. We denote by $\pi_i$ the projection onto the $i^{\text{th}}$ coordinate.

For $1\leq j\leq k$, define $\vec{t}_j:=(t_1,\ldots,t_j)$.  A straightforward induction shows that for all $n\in\N \cup \{0\}$,
\begin{align}
\label{eqn_skew_map_calculation_for_iterated_id}
    \begin{aligned}T^n_{h,\id,\ldots,\id}&(x,\vec{t}_k) = \\
    &\big(T^n x, \ p_{\vec{t}_1}(n)+h_{1,n}(x), \ p_{\vec{t}_2}(n)+h_{2,n}(x), \ \ldots, \ p_{\vec{t}_k}(n)+h_{k,n}(x) \big),\end{aligned}
\end{align}
where we define, for $1\leq m\leq n-1$, $h_{0,m} \defeq h\circ T^m$, and, for $1\leq j\leq k,$
\begin{align*}
    h_{j,i} &\defeq \begin{cases} 0 &  0\leq i\leq j-1 \\ \sum_{m=0}^{i-1}h_{j-1,m} &   j\leq i\leq n \end{cases} \ \text{ and } \
    p_{\vec{t}_j}(n) \defeq \sum_{i=0}^{j-1} \binom{n}{i} t_{j-i}.
\end{align*}
Note that $h_{1,n}$ is equal to $h_n$, which is previously established notation that we will continue to use.

The proof of \cref{maintheorem_combinatorial_form} requires two lemmas.  The first describes a sequence with an almost periodic $k^{\text{th}}$ derivative as the last coordinate in the trajectory of a point in a system of the form $(X \times \T^k, T_{h,\id,\ldots,\id})$.  The second demonstrates that the minimality of systems of the form $(X \times \T^k, T_{h,\id,\ldots,\id})$ depends only on the minimality of the initial skew-product system $(X \times \T, T_{h})$.

\begin{lemma}
\label{claim_bohr_ap_form_for_iterated_difference}
Let $f: \Z \to \T$ and $k \in \N$.  If $\Delta_1^k f$ is almost periodic, then there exists a minimal equicontinuous system $(X,T)$, a continuous function $h: X \to \T$, and a point $x \in X$ such that for all $n \in \N\cup\{0\}$, $\Delta_1^k f(n) =h(T^n x)$.  Moreover, the points $t_i \defeq \Delta_1^{k-i}f(0) \in \T$, $1 \leq i \leq k$, are such that for all $n \in \N\cup\{0\}$,
\begin{align}
\label{eqn_f_expressed_as_last_coord_in_iterated_skew}
    f(n)=\pi_{k+1}\Big(T_{h,\id,\ldots,\id}^n \big(x,t_1, \ldots, t_k \big)\Big).
\end{align}
\end{lemma}

\begin{proof}
The statement in \cref{lemma_equicontinuous_observables_are_ap} continues to hold when both instances of $\N$ are replaced by $\Z$, noting that equicontinuous systems are invertible. Thus, if $\Delta_1^k f$ is almost periodic, the first assertion in the statement of the lemma follows immediately from \cref{lemma_equicontinuous_observables_are_ap}.

For $1 \leq j \leq k$, define $t_{j,f}^{(k)} \defeq \Delta_1^{k-j}f(0)$.  We will prove \eqref{eqn_f_expressed_as_last_coord_in_iterated_skew} by induction on $k$, proving that if for all $n \in \N\cup\{0\}$, $\Delta_1^{k}f(n) = h(T^nx)$, then for all $n\in\N\cup\{0\},$
\begin{align}
\label{eqn_induction_to_show_f_sequence_in_last_coord}
    f(n)=\pi_{k+1}\Big(T_{h,\id,\ldots,\id}^n \big(x,t_{1,f}^{(k)},\ldots,t_{k,f}^{(k)} \big)\Big).
\end{align}

When $k=1,$ summing the relation
\[f(i+1)-f(i) = \Delta_1 f(i)=h(T^{i}x)\] 
on $i$ from $0$ to $n-1$, we see that 
\[f(n)=f(0)+\sum_{i=0}^{n-1}h(T^{i}x) = t_{1,f}^{(1)}+h_{1,n}(x)=\pi_2\Big(T^n_{h}\big(x,t_{1,f}^{(1)} \big)\Big).\]
This verifies \eqref{eqn_induction_to_show_f_sequence_in_last_coord} in the base case.

Assume now that \eqref{eqn_induction_to_show_f_sequence_in_last_coord} holds for some $k \in \N$ and that for all $n \in \N\cup\{0\}$, $\Delta^{k+1}_1 f(n)=h(T^n x)$. Since $\Delta_1^{k+1} = \Delta_1^k \Delta_1 f$, we can appeal to the induction hypothesis for the function $\Delta_1 f$ to see that for all $n \in \N\cup\{0\}$,
\[\Delta_1f(n) = \pi_{k+1}\Big(T^n_{h,\id,\ldots,\id}\big(x,t_{1,\Delta_1f}^{(k)},\ldots,t_{k,\Delta_1f}^{(k)} \big)\Big).\]
Note that $t_{j,\Delta_1f}^{(k)}=\Delta^{k-j}_1 \Delta_1 f(0)=t_{j,f}^{(k+1)}$. Recall that $\vec{t}_{j}=(t_1,\ldots,t_j)$ and, in analogy, define $\vec{t}_{j,f}^{\,(k+1)}= \big(t_{1,f}^{(k+1)},\ldots,t_{j,f}^{(k+1)} \big)$. By the definition of the polynomials $p_{\vec{t}_j}$,
\begin{align}
    \label{eqn_step_up_eqn_for_polynomials}
    \begin{aligned}
    \sum_{i=0}^{n-1}p_{\vec{t}_{k,f}^{~(k+1)}}(i) &= \sum_{i=0}^{n-1}\sum_{j=0}^{k-1}\binom{i}{j}t_{k-j,f}^{(k+1)} \\
    &= \sum_{j=1}^{k}\binom{n}{j}t_{k+1-j,f}^{(k+1)} = p_{\vec{t}_{k+1,f}^{~(k+1)}}(n) -t_{k+1,f}^{(k+1)}.
    \end{aligned}
\end{align}    
 Summing the relation
\[f(i+1)-f(i)= \Delta_1f(i) = \pi_{k+1}\Big(T^i_{h,\id,\ldots,\id}\big(x,t_{1,f}^{(k+1)},\ldots,t_{k,f}^{(k+1)} \big)\Big)\]
on $i$ from $0$ to $n-1$ and appealing twice to \eqref{eqn_skew_map_calculation_for_iterated_id} and once to \eqref{eqn_step_up_eqn_for_polynomials}, we see that
\begin{align*}
f(n) & = f(0)+\sum_{i=0}^{n-1}\pi_{k+1}\Big(T^n_{h,\id,\ldots,\id}\big(x,t_{1,f}^{(k+1)},\ldots,t_{k,f}^{(k+1)} \big)\Big)\\
& = \Delta_1^0 f(0)+\sum_{i=0}^{n-1}\Big(p_{\vec{t}_{k,f}^{\,(k+1)}}(i)+h_{k,i}(x)\Big)\\
& = t_{k+1,f}^{(k+1)}+\sum_{i=0}^{n-1}p_{\vec{t}_{k,f}^{\,(k+1)}}(i)+h_{k+1,n}(x)\\
& = p_{\vec{t}_{k+1,f}^{\,(k+1)}}(n)+h_{k+1,n}(x)\\
& = \pi_{k+2}\Big(T^n_{h,\id,\ldots,\id}\big(x,t_{1,f}^{(k+1)},\ldots,t_{k+1,f}^{(k+1)}\big)\Big).
\end{align*}
This verifies the inductive step, proving \eqref{eqn_induction_to_show_f_sequence_in_last_coord} and concluding the proof of the lemma.
\end{proof}

\begin{lemma}
\label{lemma_minimality_of_higher_skew_given_base_minimality}
Let $(X, T)$ be a system, and let $h: X \to \T$ be continuous.  If the skew-product system $(X \times \T, T_h)$ is minimal, then for all $k \in \N$, the skew-product system $(X \times \T^{k}, T_{h, \id, \ldots, \id})$ is minimal.
\end{lemma}

\begin{proof}
The $k=1$ case is tautologically true. For the remaining cases, it suffices by induction on $k$ to show the $k = 2$ case: the system $(X \times \T^2, T_{h, \id})$ is minimal.  To see why, note that for $k \geq 3$, the system $(X \times \T^{k}, \allowbreak T_{h, \id, \ldots, \id})$ can be written as $(X' \times \T^2, T'_{h', \id})$, where $X' = X \times \T^{k-2}$, $T' = T_{h, \id, \ldots, \id}$ (where the identity skewing map is repeated $k-3$ many times), and $h': X' \to \T$ is defined by $h'(x, t_1, \ldots, t_{k-2}) = t_{k-2}$.  The induction hypothesis gives the minimality of the system $(X' \times \T, T'_{h'})$, and then the $k=2$ argument below gives the minimality of $(X' \times \T^2, T'_{h', \id})$.

To see that $(X \times \T^2, T_{h, \id})$ is minimal, fix $x_0 \in X $, and put \[Z = \overline{ \big\{ T_{h,\id}^n (x_0,0,0) \ | \ n \in \N \big\} }.\]  We will show that $Z = X \times \T^2$.  This proves that $(X \times \T^2, T_{h, \id})$ is minimal.  Indeed, because the system $(X \times \T, T_h)$ is minimal, any minimal subsystem $Z'$ of $(X \times \T^2, T_{h, \id})$ contains a point of the form $(x_0,0,s)$, $s \in \T$.  Applying the automorphism $\psi: (x,t_1,t_2) \mapsto (x,t_1,t_2 - s)$ to $Z'$, we find that $\psi Z'$ contains the point $(x_0,0,0)$ and hence $\psi Z'=X \times \T^2$.  It follows that $Z' = X \times \T^2$.

Consider the following statements:
\begin{enumerate}
    \item \label{item_minimalityofmultipleskew_one} for all $\eps > 0$, there exists $x_1 \in X$, an $\eps$-dense set $Y \subseteq \T$, and $z_1 \in \T$ such that $\{x_1\} \times Y \times \{z_1\} \subseteq Z$;
    \item \label{item_minimalityofmultipleskew_two} there exists $x_1 \in X$ and $z_1 \in \T$ such that $\{x_1\} \times \T \times \{z_1\} \subseteq Z$;
    \item \label{item_minimalityofmultipleskew_three} for all $\eps > 0$, there exists $x_1 \in X$ and an $\eps$-dense set $W \subseteq \T^2$ such that $\{x_1\} \times W \subseteq Z$;
    \item \label{item_minimalityofmultipleskew_four} there exists $x_1 \in X$ such that $\{x_1\} \times \T^2 \subseteq Z$.
\end{enumerate}
We will prove \eqref{item_minimalityofmultipleskew_one}, then show that \eqref{item_minimalityofmultipleskew_one} implies \eqref{item_minimalityofmultipleskew_two} implies \eqref{item_minimalityofmultipleskew_three} implies \eqref{item_minimalityofmultipleskew_four}. Because $(X \times \T, T_h)$ is minimal, the system $(X,T)$ is minimal, and so \eqref{item_minimalityofmultipleskew_four} implies that $Z = X \times \T^2$, concluding the proof of the lemma.

To prove that \eqref{item_minimalityofmultipleskew_one} holds, let $\eps > 0$. Let $\alpha_1, \ldots, \alpha_N \in \T$ be such that the set $\{1, \alpha_1, \ldots, \alpha_N\}$ is linearly independent over $\Q$ and the set $\{\alpha_1, \ldots, \alpha_N\}$ is $\eps$-dense in $\T$.  Because $(X \times \T, T_h)$ is minimal, for each $i \in \{1, \ldots, N\}$, there exists $t_i \in \T$ such that $(x_0,\alpha_i, t_i) \in Z$. Since $\{1, \alpha_1, \ldots, \alpha_N\}$ is linearly independent over $\Q$, there exists $(n_\ell)_{\ell =1}^\infty \subseteq \N$ such that
\[\lim_{\ell \to \infty} n_\ell \big( \alpha_1, \ldots, \alpha_N \big) = (-t_1, \ldots, -t_N) \in \T^N.\]

Using the sequential compactness of $X$ and $\T$, by passing to a subsequence of $(n_\ell)_{\ell =1}^\infty$, we may assume without loss of generality that there exists $x_1 \in X$ and $y_1, z_1 \in \T$ such that for all $i \in \{1, \ldots, N\}$, $\lim_{\ell \to \infty} T_h^{n_\ell} (x_0, \alpha_i) = (x_1, y_1 + \alpha_i)$, and such that $\lim_{\ell \to \infty} \sum_{j=0}^{n_\ell - 1} h_j(x_0) =  z_1$. Note then that for $i \in \{1, \ldots, N\}$,
\begin{align*}
    \lim_{\ell \to \infty} T^{n_\ell}_{h,\id} (x_0, \alpha_i, t_i) &= \lim_{\ell \to \infty}\left( T^{n_\ell}_h (x_0, \alpha_i), t_i + n_\ell \alpha_i + \sum_{j=0}^{n_\ell-1} h_j(x_0) \right)\\
    &= (x_1, y_1 + \alpha_i, z_1).
\end{align*}
Because $(x_0, \alpha_i, t_i) \in Z$ and $Z$ is closed and $T_{h,\id}$-invariant, we get that $(x_1, y_1 + \alpha_i, z_1) \in Z$.  Defining $Y = y_1 + \{\alpha_1, \ldots, \alpha_N\}$, we see that $\{x_1\} \times Y \times \{z_1\} \subseteq Z$, verifying \eqref{item_minimalityofmultipleskew_one}.

By the compactness of $X$, $\T$, and $Z$, it is immediate that \eqref{item_minimalityofmultipleskew_one} implies \eqref{item_minimalityofmultipleskew_two} and that \eqref{item_minimalityofmultipleskew_three} implies \eqref{item_minimalityofmultipleskew_four}.  Thus, we have only left to show that \eqref{item_minimalityofmultipleskew_two} implies \eqref{item_minimalityofmultipleskew_three}.

Suppose that \eqref{item_minimalityofmultipleskew_two} holds: there exists $x_1 \in X$ and $z_1 \in \T$ such that $\{x_1\} \times \T \times \{z_1\} \subseteq Z$.  Let $\eps > 0$.  Note that for $n \in \N$ and $y \in \T$, because $Z$ is $T_{h,\id}$-invariant,
\begin{align*}
    T^{n}_{h,\id} (x_1, y, z_1) = \left( T^n x_1, y + h_n(x_1), z_1 + ny + \sum_{i=0}^{n-1} h_i(x_1) \right) \in Z.
\end{align*}
Let $n \in \N$ be sufficiently large so that $\big\{ (y, ny) \in \T^2 \ \big | \ y \in \T \big\}$ is $\eps$-dense in $\T^2$. Recalling now that $x_1$, $z_1$, and $n$ are fixed, define
\[W \defeq \left\{ \left(y + h_n(x_1), ny + z_1 + \sum_{i=0}^{n-1} h_i(x_1) \right) \in \T^2 \ \middle | \ y \in \T \right\},\]
and note that $W$ is $\eps$-dense in $\T^2$.  It follows from the calculation of $T^{n}_{h,\id} (x_1, y, z_1)$ above that $\{T^n x_1\} \times W \subseteq Z$, verifying \eqref{item_minimalityofmultipleskew_three}, and finishing the proof of the lemma.
\end{proof}

\begin{proof}[Proof of \cref{maintheorem_combinatorial_form}]
Let $k \in \N \cup \{0\}$ be such that $\Delta_1^k f$ is almost periodic, and let $\eps > 0$.  If $k = 0$, then $\Delta_1^0 f = f$ is almost periodic.  The set $A$ contains the set of $\eps$-almost periods of $f$, hence $A$ is a \bz{} set.  It follows immediately that $A$ is syndetic and that $A-A$ contains a Bohr neighborhood of zero.

Suppose $k \geq 1$.  Since the conclusion of the theorem does not make reference to $k$, we may assume without loss of generality that $k$ is the least positive integer such that $\Delta_1^k f$ is almost periodic; in other words, we may assume that $\Delta_1^{k-1} f$ is not almost periodic.

By \cref{claim_bohr_ap_form_for_iterated_difference}, there exists a minimal equicontinuous system $(X,T)$, a continuous map $h: X \to \T$, and a point $x \in X$ such that for all $n \in \N\cup\{0\}$, $\Delta_1^k f(n) =h(T^n x)$.  Moreover, the points $t_i \defeq \Delta_1^{k-i}f(0) \in \T$, $1 \leq i \leq k$, are such that for all $n \in \N\cup\{0\}$, the equality in \eqref{eqn_f_expressed_as_last_coord_in_iterated_skew} holds.

Summing $\Delta_1^k f(n) =h(T^n x)$, we see that for all $n \in \N\cup\{0\}$,
\[\Delta_1^{k-1}f(n) - t_1 = \sum_{i = 0}^{n-1} \Delta_1^k f(i) = \sum_{i = 0}^{n-1} h(T^i x) = h_n(x).\]
Therefore, for all $n \in \N\cup\{0\}$,
\[\Delta_1^{k-1}f(n) = t_1 + h_n(x) = \pi_2 \big( T^n_h (x,t_1) \big).\]
This calculation shows that if the system $(X \times \T, T_h)$ were equicontinuous, the sequence $\Delta_1^{k-1}f$ would be almost periodic. Since $\Delta_1^{k-1}f$ is not almost periodic, the system $(X \times \T, T_h)$ is not equicontinuous.  It follows by \cref{theorem_not_minimal_implies_equicontinuous} that the system $(X \times \T, T_h)$ is minimal, and then it follows by \cref{lemma_minimality_of_higher_skew_given_base_minimality} that the system $(X \times \T^k, T_{h, \id, \ldots, \id})$ is minimal.  

Since the system $(X \times \T^k, T_{h, \id, \ldots, \id})$ is minimal, the set
\[B \defeq \big\{ n \in \N \ \big| \ d_{X \times \T^k} \big(T^n_{h,\id,\ldots,\id}(x,t_1,\ldots,t_k),(x,t_1,\ldots,t_k) \big) < \eps \big\}\]
is syndetic.  Since \eqref{eqn_f_expressed_as_last_coord_in_iterated_skew} holds and $t_k = f(0)$, we see that $B \subseteq A$, proving that the set $A$ is syndetic.

Let $U$ be the open $\eps$-ball about the point $(x,t_1, \ldots, t_k)$ in $X \times \T^k$. Since the system $(X \times \T^k, T_{h, \id, \ldots, \id})$ is minimal, it follows by a standard lemma (see, for example, \cite[Lemma 4.4]{boshernitzanglasner2009}) that
\[\big\{ n \in \N \ \big| \ U \cap T_{h, \id, \ldots, \id}^{-n} U \neq \emptyset \big\} = B - B.\]
It follows then from Lemma 2.12 that there exists $\delta > 0$ such that $\dyret_\delta(X \times \T^k, T_{h, \id, \ldots, \id}) \subseteq B-B$.  By Theorem A, there is a positive answer to Katznelson's Question for the system $(X \times \T^k, T_{h, \id, \ldots, \id})$: the set $\dyret_\delta(X \times \T^k, T_{h, \id, \ldots, \id})$ is a \bz{} set.  It follows that $B-B$, and hence $A-A$, is a \bz{} set, as was to be shown.
\end{proof}

\begin{example}
\label{example_application_of_thm_c}
Here's an example application of \cref{maintheorem_combinatorial_form} that doesn't seem to follow easily by other means.  Define $\varphi: \N \cup \{0\} \to \R$ by
\[\varphi(n) = \sum_{i=0}^\infty \frac{d_i(n)}{2^i},\]
where $d_i(n) \in \{0,1\}$ is the $i^{\text{th}}$ least-significant digit of $n$ in binary.  We claim that for all $\eps > 0$, the set
\[A \defeq \big\{ n \in \N \ \big| \ \| \varphi(1) + \cdots + \varphi(n) \| < \eps \big\}\]
is syndetic and is such that $A-A$ is a \bz{} set.  To see how this follows from \cref{maintheorem_combinatorial_form}, note that $\varphi$ is uniformly continuous with respect to the 2-adic metric on $\Z$, hence it extends to a continuous function $\varphi: \Z_2 \to \R$, where $\Z_2$ denotes the $2$-adic integers. Defining $T: \Z_2 \to \Z_2$ to be addition by 1, we see that the sequence $n \mapsto \varphi(n) = \varphi(T^n 0)$ is almost periodic.  Define $f: \N \cup \{0\} \to \R$ by $f(n) = \sum_{i=0}^{n-1} \varphi(i)$, so that $\Delta_1 f = \varphi$ is almost periodic. Invoking \cref{maintheorem_combinatorial_form} for the function $\pi \circ f$, where $\pi: \R \to \T$ is the quotient map, we see that $A$ is syndetic and that $A-A$ is a \bz{} set.
\end{example}

\section{Open questions}
\label{sec_open_questions}

The open questions in this section are split between those that are natural extensions of results in this work (in \cref{sec_next_steps}) and those that are closely related to the combinatorial form of Katznelson's Question (in \cref{sec_katznelson_adjacent}).

\subsection{Next steps}
\label{sec_next_steps}

The result in
\cref{maintheorem_katznelson_for_skew_towers} leads one naturally to ask whether more general extensions of equicontinuous systems have \bz{} large returns. We record two questions in this vein here.

\begin{question}
\label{question_basic_skew_over_2_step_nil}
Let $\alpha \in \T \setminus \Q$ and $h: \T \to \T$ be continuous. Does the skew product system $(\T^3, T)$ given by
\begin{align}
\label{eqn_two_skew_system}
    T(x,y,z) = \big(x+ \alpha, y+x, z + h(x) \big)
\end{align}
have \bz{} large returns?
\end{question}

Because the skewing function $x \mapsto (x, h(x))$ from the system in \eqref{eqn_two_skew_system} maps into $\T^2$, the intermediate value theorem technique at the heart of \cref{maintheorem_katznelson_for_skew_towers} does not seem to help in answering \cref{question_basic_skew_over_2_step_nil}.  More generally, one can ask about skew products by the 2-torus over equicontinuous systems.  Analogously, we can ask whether or not the result in \cref{maintheorem_combinatorial_form} continues to hold when $\T$ is replaced by $\T^2$.

In the following question, by an \emph{automorphism of $(X,T)$}, we mean a homeomorphism $\varphi: X \to X$ such that $\varphi \circ T = T \circ \varphi$.  If a compact group $K$ acts on $(X,T)$ by automorphisms, the \emph{quotient system} is composed of the set of equivalence classes $X/K \defeq \{Kx \ | \ x \in X\}$, a compact metric space, and the map $T$, which descends to a continuous self map of $X / K$.

\begin{question}
\label{question_more_general_isometric_ext}
Let $(X,T)$ be a minimal system.  Suppose that $\T$ acts on $(X,T)$ by automorphisms in such a way that the quotient system $(X / \T, T)$ is equicontinuous.  Does the system $(X,T)$ have \bz{} large returns?
\end{question}

A negative answer to either of the previous questions would yield a negative answer to the Katznelson's Question.  Positive answers, on the other hand, would represent significant progress in the topological-dynamical, structure-theoretic approach to resolving it.

\begin{question}
\label{question_3_syndetic_sets}
Is it true that for all $3$-colorings $\N = A_1 \cup A_2 \cup A_3$, the set \[(A_1 - A_1)\cup(A_2 - A_2)\cup(A_3 - A_3)\]
is a \bz{} set?
\end{question}

In the same way as \cref{cor_differences_of_a_two_syndetic_set} follows from \cref{theorem_two_coloring_Katznelson}, a positive answer to \cref{question_3_syndetic_sets} would imply that the set of pairwise differences of a 3-syndetic set is a \bz{} set.

It is also natural to ask for analogues of the sequential form of Katznelson's Question in which we consider more general notions of almost periodicity.  A sequence $f: \Z \to \R$ is \emph{Besicovitch almost periodic} if for all $\eps > 0$, there exists an almost periodic sequence $a: \Z \to \R$ such that
\[\limsup_{N \to \infty} \frac 1{2N+1} \sum_{n=-N}^N \big| f(n) - a(n) \big| < \eps.\]
It is a short exercise to show that if $f$ is Besicovitch almost periodic, then 
\begin{align}
\label{eqn_set_in_sequences_question_BAP}
    \big\{ m \in \Z \ | \ \inf_{n \in \Z} \| f(n+m) - f(n) \| < \eps \big\}
\end{align}
contains a Bohr neighborhood of zero. Is the same true for sequences whose first derivatives are Besicovitch almost periodic?

\begin{question}
\label{question_bap_sequence_extension}
If $f: \Z \to \R$ is such that $\Delta_1 f$ is Besicovitch almost periodic, does the set in \eqref{eqn_set_in_sequences_question_BAP} contain a Bohr neighborhood of zero?
More generally, if $\Delta_1^k f$ is Besicovitch almost periodic for some $k\in\N$, then does the set in \eqref{eqn_set_in_sequences_question_BAP} contain a Bohr neighborhood of zero?
\end{question}

\subsection{Closely related questions}
\label{sec_katznelson_adjacent}

The question was raised in \cite[Problem 2.2]{ruzsa_difference_sets_unpublished_1985} (and again more recently in \cite{bergelson_ruzsa_2009}) as to the existence of a set of positive upper asymptotic density whose set of differences does not contain a Bohr set. Griesmer \cite{griesmer_separating_bohr_denseness_2020} showed that such sets do exist.  The following question is an analogue for syndetic sets that remains unanswered.

\begin{question}
\label{question_differences_contain_bohr}
If $A \subseteq \N$ is syndetic, is the set $A-A$ a Bohr set?
\end{question}

An affirmative answer to the combinatorial form of Katznelson's Question clearly yields an affirmative answer to \cref{question_differences_contain_bohr}, but we were not able to answer Katznelson's Question by assuming a positive answer to \cref{question_differences_contain_bohr}.

An inhomogeneous result achieved in \cite{BFW2006} gives that the set $A-B$ is piecewise Bohr (a Bohr set intersected with a set containing arbitrarily long intervals) when $A$ and $B$ have positive upper Banach density.  Along those lines, it is natural to ask about an asymmetric result for differences of syndetic sets.  Interestingly, \cref{question_differences_contain_bohr} is equivalent to the following asymmetric one:
\begin{quote}
\emph{If $A, B \subseteq \N$ are syndetic, is the set $A-B$ a Bohr set?}
\end{quote}
To see that this is the same question as \cref{question_differences_contain_bohr}, it is easiest to see that both questions are equivalent to a third:
\begin{quote}
    \emph{If $A \subseteq \N$ is piecewise syndetic, is the set $A-A$ a Bohr set?}
\end{quote}
This question is equivalent to \cref{question_differences_contain_bohr} by the same reasoning that shows that Questions C1 and C2 are equivalent.  Clearly a positive answer to the inhomogeneous question yields a positive answer to \cref{question_differences_contain_bohr}.  Conversely, suppose $A$ and $B$ are syndetic.  Since $B$ is syndetic, there exists $k \in \N$ such that $A \subseteq \cup_{i=1}^k \big(A \cap (B-i) \big)$.  By the partition regularity of piecewise syndeticity, there is $i \in \{1, \ldots, k\}$ such that $A \cap (B-i)$ is piecewise syndetic.  If the set of differences of $A \cap (B-i)$ contains a Bohr set, then so does the set $A-B$.

To the authors' knowledge, the following analogue to \cref{question_differences_contain_bohr} concerning sumsets is also open.

\begin{question}
\label{question_sumset_contains_bohr}
If $A \subseteq \N$ is syndetic, is the set $A+A$ a Bohr set?
\end{question}

The inhomogeneous version of \cref{question_sumset_contains_bohr} -- \emph{If $A, B \subseteq \N$ are syndetic, does $A+B$ contain a Bohr set?} -- is also unanswered.  It can be shown that if $C \subseteq \N$ is syndetic, then there exist $A, B \subseteq \N$ syndetic such that $A+B \subseteq C-C$. Therefore, if there is a positive answer to the inhomogeneous analogue of \cref{question_sumset_contains_bohr}, then there is a positive answer to  \cref{question_differences_contain_bohr}. Beyond this, the relationship between Questions \ref{question_differences_contain_bohr} and \ref{question_sumset_contains_bohr} is not clear.
Interestingly, if $A=\{n\in\N: \|n^2\alpha \|<\eps\}$, then $A+A$ is a Bohr set, but not a $\bz$ set.

For $A \subseteq \N$, note that $n \in A-A \text{ if and only if } A \cap (A-n) \neq \emptyset$. This relation helps to motivate the next question, a higher-order analogue of the combinatorial form of Katznelson's Question.  A \emph{nil$_k$-\bz{} set} is the set of return times of a point to a neighborhood of itself in a $k$-step nilsystem; the curious reader is pointed toward \cite{huang_shao_ye_nilbohr2016} for more information.

\begin{question}
\label{quest_higher_order_difference_sets}
If $A \subseteq \N$ is syndetic, is it true that for all $k \in \N$, the set
\[\big\{ n \in \N \ \big| \ A \cap (A-n) \cap \cdots \cap (A-kn) \neq \emptyset \big\}\]
is a nil$_k$-\bz{} set?
\end{question}

The sets appearing in \cref{quest_higher_order_difference_sets} are intimately related to times of multiple recurrence in dynamical systems, so dynamical analogues of \cref{quest_higher_order_difference_sets} can be naturally formulated in a way that parallels the relation between Question C1 and Questions D1 and D2; see \cite[Proposition 2.3.4]{huang_shao_ye_nilbohr2016}.

A set $A \subseteq \Z^d$ is \emph{syndetic} if finitely many translates of $A$ cover $\Z^d$.   The following question is a generalization of the combinatorial form of Katznelson's Question.  It can be formulated more generally in locally compact abelian groups using characters and in more general topological groups using the Bohr compactification; see \cite{landstad1971}.

\begin{question}
\label{quest_z2_katznelson}
If $A \subseteq \Z^d$ is syndetic, does its set of pairwise differences $A-A$ contain a set of the form
\[\big\{ z \in \Z^d \ \big| \ z \cdot \lambda_1, \ldots, z \cdot \lambda_k \in U \big\}\]
where $\lambda_1, \ldots, \lambda_k \in \T^d$ and $U \subseteq \T$ is an open neighborhood of 0?
\end{question}

\bibliographystyle{alpha}
\bibliography{katzbib}

\end{document}